\documentclass{amsart}

\allowdisplaybreaks

\def\path{{\tt path}}

\newcommand{\ccal}{\mathcal{C}}

\newcommand{\fcal}{\mathcal{F}}

\newcommand{\rcal}{\mathcal{R}}
\newcommand{\scal}{\mathcal{S}}
\newcommand{\tcal}{\mathcal{T}}

\newcommand{\vcal}{\mathcal{V}}

\newcommand{\kcal}{\mathcal{K}}

\newcommand{\real}{\mathbb{R}}
\newcommand{\cir}{\mathbb{T}}

 \newcommand{\ga}{\alpha}
 \newcommand{\gb}{\beta}
 \newcommand{\gc}{\gamma}
 
 \newcommand{\gd}{\delta}
 \newcommand{\gD}{\Delta}
 \newcommand{\gee}{\epsilon}

 \newcommand{\gl}{\lambda}

 \newcommand{\go}{\omega}

 \newcommand{\gs}{\sigma}

\newcommand{\eps}{\varepsilon}

\newcommand{\id}{\mathrm{id}}

\newcommand{\bdm}{\begin{displaymath}}
\newcommand{\edm}{\end{displaymath}}
\newcommand{\bea}{\begin{eqnarray*}}
\newcommand{\eea}{\end{eqnarray*}}
\newcommand{\bean}{\begin{eqnarray}}
\newcommand{\eean}{\end{eqnarray}}

\newcommand{\prob}{\mathbb{P}}

\newcommand{\E}{\mathbb{E}}

\newcommand{\G}{\widetilde G}
\newcommand{\p}{\widetilde p}
\newcommand{\wt}{\widetilde}

\newcommand{\x}{{\mathbf{x}}}
\newcommand{\y}{{\mathbf{y}}}
\newcommand{\bxi}{{\boldsymbol{\xi}}}
\newcommand{\bzeta}{{\boldsymbol{\zeta}}}
\newcommand{\bfeta}{\boldsymbol{\eta}}

\theoremstyle{plain} \newtheorem{theorem}{Theorem}[section]
\theoremstyle{plain} \newtheorem{proposition}[theorem]{Proposition}
\theoremstyle{plain} \newtheorem{corollary}[theorem]{Corollary}
\theoremstyle{plain} \newtheorem{lemma}[theorem]{Lemma}
\theoremstyle{definition} \newtheorem{definition}[theorem]{Definition}
\theoremstyle{definition} \newtheorem{assumption}[theorem]{Assumption}
\theoremstyle{remark}
\theoremstyle{remark}

\numberwithin{equation}{section}

\begin{document}

\title[Coalescing particle systems]{Coalescing systems of \\
non-Brownian particles}

\author{Steven N. Evans}
\email{evans@stat.Berkeley.EDU}

\address{Department of Statistics \#3860 \\
 University of California at Berkeley \\
367 Evans Hall \\
Berkeley, CA 94720-3860 \\
U.S.A.}

\author{Ben Morris}
\email{morris@math.ucdavis.edu}

\address{Department of Mathematics \\
University of California at Davis \\
Mathematical Sciences Building \\
One Shields Avenue \\
Davis, CA 95616 \\
U.S.A.}

\author{Arnab Sen}
\email{arnab@stat.Berkeley.EDU}

\address{Department of Statistics \#3860 \\
 University of California at Berkeley \\
367 Evans Hall \\
Berkeley, CA 94720-3860 \\
U.S.A.}

\thanks{SNE supported in part by NSF grants DMS-0405778 and DMS-0907630}
\thanks{BM supported in part by NSF grant DMS-0707144}

\date{\today}



\begin{abstract}
A well-known result of Arratia shows that one can make rigorous the notion 
of starting an independent Brownian motion at every point of an arbitrary closed subset of
the  real line and then building a set-valued process by requiring particles to
coalesce when they collide.  
Arratia noted that the value of this process will be almost
surely a locally finite set at all positive times, 
and a finite set almost surely
if the initial value is compact: the key to both of these facts is the
observation that, because of the topology of the real line and the continuity
of Brownian sample paths, 
at the time when two particles collide one or the other 
of them must have already collided with each particle 
that was initially between them.   
We investigate whether such
instantaneous coalescence 
still occurs for coalescing systems of particles where either the 
state space of the individual particles is not locally homeomorphic to an
interval or the sample paths of the individual particles are discontinuous.
We give a quite general criterion for a coalescing system of
particles on a compact state space to coalesce to a finite
set at all positive times almost surely
and show that there is almost sure
instantaneous coalescence to a locally finite set 
for systems of Brownian motions on the Sierpinski gasket
and stable processes on the real line with stable index greater than one.
\end{abstract}
\maketitle

\section{Introduction}

A construction due to Richard Arratia \cite{arratia79, arratia81}
shows that it is possible to make rigorous
sense of the informal notion of starting an independent Brownian motion at each
point of the real line and letting particles coalesce when they collide.

Arratia proved that the set of particles remaining
at any positive time is locally finite almost surely.    
Arratia's argument is based on the simple
observation that at the time two particles collide, one or the other
must have already collided with each particle that was initially between them.
The same argument shows that if we start an independent circular
Brownian motion at each
point of the circle and let particles coalesce when they collide,
then, almost surely, there are only finitely many particles remaining
at any positive time.

Arratia established something even stronger: 
it is possible to construct a flow of random maps 
$(F_{s,t})_{s < t}$ from the real line
to itself in such a way that for each fixed $s$ the process 
$(F_{s,s+u})_{u \ge 0}$
is given by the above particle system.  
Arratia's flow has since been studied by several
authors such as  \cite{MR1640799, MR1785393, 
MR1917545, MR2060298, MR2079671, MR2094432, MR2519355} for purposes as
diverse as giving a rigorous definition of 
a one-dimensional self-repelling Brownian
motion to providing examples of noises that are, in some sense, completely ``orthogonal'' to those
produced by Poisson processes or Brownian motions.

Coalescing systems of more general Markov processes have been investigated
because of their appearance as the duals of models in genetics of the stepping stone type,
see, for example,
\cite{MR1404524, MR1415234, MR1457055,  MR1797304, 
MR1986843, MR2113855, MR2162813, MR2259212, MR2384748}.

We show in Section~\ref{S:countable} that if
$E$ is a locally compact, second-countable, 
Hausdorff (and hence metrizable) space 
and $X$ is a Feller process on $E$, then it is possible to define
a process $\bzeta$ taking values in $E^{\mathbb{N}}$ with the property that
the coordinate processes evolve as independent copies of $X$ until two 
such processes collide, after which those two coordinate processes evolve
as a common copy of $X$.  Write $\Xi_t$ for the closure
of the random countable set $\{\zeta_i(t) : i \in \mathbb{N}\}$.  
We demonstrate that if $\x'$ and $\x''$ 
are two elements of $E^{\mathbb{N}}$ with the same
closure, then the distribution of $\Xi$ when
$\bzeta$ starts from $\x'$ is the same as that when
$\bzeta$ starts from $\x''$, and it follows that $\Xi$
is a strong Markov process.  Taking the entries in the sequence
$\bzeta_0 \in E^{\mathbb{N}}$ to be a countable dense subset of $E$
gives $\Xi_0 = E$ and corresponds to the intuitive idea of constructing
a coalescing particle system with an initial
condition consisting of a particle at each point of $E$.

Arratia's ``topological'' argument for the instantaneous coalescence 
of such a system to a locally finite set fails when one considers 
Markov processes on the line or circle with discontinuous sample paths or
Markov processes with state spaces that are not locally like the real line.
We show, however, that analogous conclusions 
holds for coalescing Brownian motions on
the ``finite'' and ``infinite'' (that is, compact and non-compact) 
Sierpinski gaskets and stable processes on the circle and line -- 
provided, of course, that the stable index is greater than $1$, so that an 
independent pair of such motions collides with positive probability.

In order to motivate some of the estimates that we develop for each
of these cases, we first give a brief sketch of how our general approach
applies to the case where the underlying Markov motion $X$ is an
appropriate process on a compact state space.
Suppose, then, that the space $E$ is compact with its topology
metrized by a metric $r$.  Consider a strong
Markov process $X$ with state space $E$. 
Let $X'$ and $X''$
be independent copies of $X$ started from $x'$ and $x''$.
Assume there are constants $\beta, \alpha, p > 0$ (not depending
on $x',x''$) such that for
all $\epsilon > 0$ 
\begin{equation}
\label{hitting_lower_bound}
r(x',x'') \le \epsilon
\quad \longrightarrow \quad
\mathbb{P}\{\exists 0 \le s \le \beta \epsilon^\alpha : X_s' = X_s''\}
\ge
p
\end{equation}
 (for example, such a bound
holds when $X$ is a stable process on the circle
with stable index $\alpha > 1$).   Suppose further that 
there are  constants $C,\kappa > 0$ such that for all subsets
$A \subseteq E$
\begin{equation}
\label{Lebesgue_number}
\# A > n
 \quad \longrightarrow \quad
r(x',x'') \le C n^{-\kappa} \text{ for some } x',x'' \in A, \, x' \ne x'', 
\end{equation}
(for example, $\kappa = 1$ for the circle).

If we start with $n+1$ particles in some configuration on $E$,
then there are at least two particles within distance at most
$C n^{-\kappa}$, and,
with probability at least $p$, these two particles in isolation
collide with each other by time $\beta C^{-\alpha} n^{-\kappa \alpha}$. 
Hence, in the coalescing system
the probability that there is at least one collision between some
pair of particles within the time interval  
$[0, \beta C^{-\alpha}  n^{-\kappa\alpha}] $ is 
certainly at least $p$ (either the two distinguished particles
collide with each other and no others or some other particle(s)
collides with one or both of the distinguished particles).  
Moreover, if  there is no collision between any pair of particles after time 
$\beta C^{-\alpha} n^{-\kappa \alpha} $,  
then we can again find at time $\beta C^{-\alpha} n^{-\kappa \alpha}$ a
possibly different pair of particles that are within distance 
$C n^{-\kappa}$ from each other,  and the probability that this pair of
 particles will collide within the time interval  
$[ \beta C^{-\alpha} n^{-\kappa \alpha}, 
2\beta C^{-\alpha} n^{-\kappa \alpha} ] $ 
is again at least $p$. 
By repeating this argument and using the Markov property, we see that if we
let $\tau^{n+1}_{n}$ be the first time there are $n$ or fewer surviving particles starting from $(n+1)$ particles, then, regardless of the particular initial configuration of the
$n+1$ particles, 
\[ 
\prob\left \{ \tau^{n+1}_{n}   \ge  k  \beta C^{-\alpha} n^{-\kappa \alpha} \right \} \le (1-p)^k.
\]
In particular, the expected time needed to reduce the number of particles  from $n+1$ to $n$ or fewer is bounded above by $c n^{-\kappa \alpha}$ for a suitable constant $c$. 

Suppose now that $\kappa \alpha > 1$.
If we start with $N$ particles somewhere on $E$, then the probability that after some positive time $t$ the number of  particles remaining is greater than $m$ is, by Markov's inequality, bounded above by 
\[    
\frac 1 t  \sum_{n = m}^{N-1} \E \left[\tau^{n+1}_{n}\right] \le  \frac c t  \sum_{n = m}^{N-1} n^{-\kappa \alpha} \le  \frac {c'}{ t} m^{1-\kappa \alpha}
\]
for some constant $c'$.
The probability in question therefore converges to zero as
$N \to \infty$ and then $m \to \infty$.  It follows that, even
if we start with a coalescing particle at each point of $E$,
by time $t$ there are only finitely many particles almost surely
(see the last part of the proof of Theorem~\ref{thm_main1_gasket} 
for the proof that the convergence to zero of the given probability implies
that an infinite coalescing system coalesces to finitely many
points instantaneously).

The above argument required three ingredients, the collision probability
bound \eqref{hitting_lower_bound}, the estimate \eqref{Lebesgue_number}
that provides quantitative information on the extent to which $E$
is totally bounded, and the assumption $\alpha \kappa > 1$.  
We show in Section~\ref{S:compact_space} that
\eqref{hitting_lower_bound} follows from suitable upper and lower
bounds on the transition densities of the process with respect
to some reference probability measure $\mu$, whereas \eqref{Lebesgue_number}
follows from the assumption that the $\mu$ mass of
any ball of radius $\epsilon$ is bounded below by a constant
multiple of $\epsilon^{1/\kappa}$.   Informally,
both of these conditions hold when the state space
and the process have suitable approximate local self-similarity
properties.

The compactness, and hence total boundedness, of the state space
was crucial for the ``pigeonhole principle'' reasoning
that we used above. The same method cannot
be applied as it stands to deal with, say, 
coalescing stable processes on the real line to show a
result of the Arratia type that the set of particles
remaining at some positive time is locally finite.  The primary
difficulty is that the proof bounds the time to coalesce from some number 
of particles to a smaller number by considering 
a particular sequence of coalescent
events, and while waiting for one of these events to occur 
the particles might spread out
to such an extent that the pigeonhole argument can no longer be used.
We overcome this problem by developing a more
sophisticated  pigeonhole argument that assigns the bulk of the particles to
a collection of suitable disjoint pairs (rather than just selecting a single
suitable pair) and then employing a simple large deviation bound to
ensure that with high probability at least a 
certain fixed proportion of the pairs will
have collided over an appropriate time interval. 

We are able to carry this general approach through for
stable processes with index $\alpha > 1$ and
Brownian motions on the infinite (that is, unbounded) Sierpinski gasket.
Elements of our argument seem rather specific to these two special
cases and differ between the two cases, 
and we don't have a general method that applies to
a broad class of processes with non-compact state spaces.
It is an interesting and challenging open problem to develop
techniques for non-compact state spaces that have
wider applicability.

We note that, as well as providing an interesting test case 
of a process with continuous sample paths on a state space 
that is not locally one-dimensional but is such that
two independent copies of the process will collide with positive probability,
the Brownian motion on the Sierpinski gasket was introduced as a model
for diffusion in disordered media and it has since attracted a considerable
amount of attention.  The reader can get a feeling for this literature
by consulting some of the earlier works
such as \cite{barlow88, MR988082, barlow95} and more recent papers such
as \cite{MR2018919, MR2161694} and the references therein.

\section{Countable systems of coalescing Feller processes}
\label{S:countable}

In this section we develop some general properties of coalescing systems
of Markov processes that we will apply later to Brownian motions
on the Sierpinski gasket and stable processes on the line or circle.

\subsection{Vector-valued coalescing process}\label{subsec_collision_rule}

Fix $N \in \mathbb N \cup \{ \infty\}$, where, as usual, 
$\mathbb N$ is the set of positive integers. 
Write $[N]$ for the set $\{1, 2, \ldots, N\}$ when $N$ is finite
and for the set $\mathbb N$ when $N = \infty$.

Fix a locally compact, second-countable, Hausdorff space $E$.
Note that $E$ is metrizable. Let $d$ be a metric giving the topology on $E$.
Denote by $D := D(\mathbb{R}_+,E)$ the usual Skorokhod space of $E$-valued c\`adl\`ag paths.
Fix a bijection $\sigma: [N] \rightarrow [N]$.  We will call $\sigma$ a {\em ranking} of $[N]$.
Define a mapping $\Lambda_\gs: D^N \rightarrow D^N$ by setting
$\Lambda_\sigma \bxi = \bzeta$ for $\bxi = (\xi_1, \xi_2, \ldots) \in D^N$, 
where $\bzeta$ is defined inductively as follows. Set $\zeta_{\gs(1)} \equiv \xi_{\gs(1)}$.
For $i > 1$, set
\[
\tau_i 
:=  
\inf \Big \{t \ge 0 : \xi_{\gs(i)}(t)  \in \{\zeta_{\gs(1)}(t) , \zeta_{\gs(2)}(t), \ldots, \zeta_{\gs(i-1)}(t)\}\Big \},
\] 
with the usual convention that $\inf \emptyset = \infty$.  Put
\[ 
J_i 
:= 
\min \Big \{ j \in \{1, 2, \ldots, i-1\}: \xi_{\gs(i)}(\tau_i) = \zeta_{\gs(j)}(\tau_i) \Big \} \quad \text{if } \tau_i < \infty.
\]
For $t \ge 0$, define
\[ 
\zeta_{\gs(i)}(t) := 
\begin{cases}
\xi_{\gs(i)}(t), &  \text{if $t <\tau_i$},\\
\zeta_{\gs(J_i)}(t), & \text{if $t \ge \tau_i$}.
\end{cases}
\]
We call the map $\Lambda_\sigma$ a {\em collision rule}.  It produces a vector of ``coalescing'' paths
from of a vector of ``free'' paths: after the free paths labeled $i$ and $j$ collide, the
corresponding coalescing paths both subsequently follow either the path labeled
$i$ or the path labeled $j$, according to whether  
$\sigma(i) < \sigma(j)$ or $\sigma(i) > \sigma(j)$. 
Note for each $n < N$ that the value of  
$(\zeta_{\sigma(i)})_{1 \le i \le n}$ is unaffected by the value of 
$(\xi_{\sigma(j)})_{j > n}$.

Suppose from now on that the paths $\xi_1, \xi_2, \ldots$ are realizations of independent 
copies of a Feller Markov process $X$ with state space $E$.

{\em A priori},  the distribution  of the finite or countable coalescing system $\bzeta = \Lambda_\gs \bxi$ depends on the ranking $\sigma$. However, we have the following result, 
which is a consequence of the strong Markov property of $\bxi$ and the observation
that if we are given a bijection $\pi:[N] \to [N]$ and define 
a map $\Sigma_\pi : D^N \rightarrow D^N$ by $(\Sigma_\pi \bxi)_i
= \xi_{\pi(i)}$, $i \in [N]$, then 
$\Sigma_\pi \Lambda_\sigma = \Lambda_{\sigma \pi^{-1}} \Sigma_{\pi}$.

\begin{lemma}[\cite{arratia79, arratia81}] \label{l:exchangeability}
The distribution of $\bzeta = \Lambda_\gs \bxi$ 
is the same for all bijections $\sigma:[N] \rightarrow [N]$.
\end{lemma} 

From now on, we will, unless we explicitly state otherwise, 
take $\sigma = \id$, where $\id: [N] \rightarrow [N]$ is the identity bijection.
To simplify notation, we will write $\Lambda$ for the collision rule $\Lambda_\id$.

It is intuitively clear that the coalescing system $\bzeta$ is Markov. 
For the sake of completeness, we establish this formally in the next lemma,
the proof of which is essentially an argument from \cite{arratia79, arratia81}. 

Define the right-continuous filtration $(\fcal_t)_{t\ge 0}$  by
\[
\fcal_t := \bigcap_{\eps > 0} \sigma \{ \xi_i(s) : s \le t + \eps, i \ge 1\}.
\]

\begin{lemma} \label{l:strong_markov}
The stochastic process $\bzeta = \Lambda \bxi$ is strong  Markov with
respect to the filtration $(\fcal_t)_{t\ge 0}$.
\end{lemma}

\begin{proof}
Define maps $m: \{1,2,\ldots, N\} \times E^N \to \{1,2,\ldots, N\}$ 
and  $\Pi: E^N \times E^N \to E^N$ by setting 
$ m(i, \x) :=  \min \{j: x_j = x_i \}$
and
$\Pi(\x, \y)_i := y_{m(i, \x)}$. Note that
\[
\Pi( \Lambda \bfeta(t), \bfeta(t)) := \Lambda \bfeta(t), \quad \bfeta \in D^N, \; t \ge 0.
\]
Define a map 
$\wt \Pi : E^N \times D^N \to D^N$ by
\[ 
\wt \Pi(\x, \bfeta)(t) = \Pi(\x, \bfeta(t)), \quad \x \in E^N, \; \bfeta \in D^N, \; t \ge 0.
\]
Writing $\{\theta_s\}_{ t \ge 0}$ for the usual family of
shift operators on $D^N$, that is, $(\theta_s \bfeta)(t) = \bfeta(s+t)$, we have
\[
\theta_s \Lambda \bfeta 
= \Lambda \wt \Pi(\Lambda \bfeta(s), \theta_s \bfeta), \quad \bfeta \in D^N, \; s \ge 0.
\]

Fix a bounded measurable function on $f: D^N \to \mathbb{R}$ and
set 
\[
g(\x, \y) 
=  
\E^\y \left[f \Big( \Lambda \wt \Pi(\x, \bxi)\Big)\right]. 
\] 
Note that since the components of $\bxi$ are independent, 
if $\Pi(\x,\y) = \x$, then $g(\x,\y) = g(\x,\x)$. 
Thus, for a finite $(\fcal_t)_{t\ge 0}$ stopping time $S$ we have
from the strong Markov property of $\bxi$ that
\[
\begin{split}
\E^\x \left[f( \theta_S \Lambda \bxi ) \, \big | \, \fcal_S\right]
& = 
\E^\x  \left[f \Big( \Lambda \wt \Pi(\Lambda \bxi(S), (\theta_S \bxi))\Big)\,  \big | \, \fcal_S \right] \\
& = g(\Lambda \bxi(S), \bxi(S)) \\
& = g(\Lambda \bxi(S), \Lambda \bxi(S)) \\
& = \E^{\Lambda \bxi(S)}[f(\Lambda \bxi)], \\
\end{split}
\]
as required.
\end{proof}

\subsection{Set-valued coalescing process}

Write $\kcal = \kcal(E)$ for the set of nonempty  compact subsets of 
$E$ equipped with the usual Hausdorff metric $d_H$ defined by
\[ 
d_H(K_1, K_2) 
:= 
\inf \{ \eps>0: K_1^\eps \supseteq K_2 \text{ and } K_2^\eps \supseteq K_1 \}, 
\]
where $K^\eps := \{ y \in E: \exists x \in K,  \, d(y, x) < \eps \}$. The metric
space $(\kcal, d_H)$ is complete.  It is compact if $E$ is.

If the locally compact space $E$ is not compact, write $\ccal = \ccal(E)$ for
the set of nonempty closed subsets of $E$.  Identify the elements of $\ccal$
with their closures in the one-point compactification $\bar E$ of $E$. 
Write $d_C$ for the metric on $\ccal$ that arises from
the Hausdorff metric on the compact subsets of $\bar E$ corresponding to some
metric on $\bar E$ that induces the topology of $\bar E$.  

Let $\Xi_t \subseteq E$ denote the closure of the set  $\{\zeta_i(t): i=1, 2, \ldots \}$ in $E$, where $\bzeta = \Lambda \bxi$.

The following result is an almost immediate consequence of 
Lemma~\ref{l:exchangeability}.

\begin{lemma}\label{l:inv_dist}
If $\x', \x'' \in E^N$ are such that the sets
$\{x_i' : i \in [N]\}$ and $\{x_i'' : i \in [N]\}$ are equal,
then the distributions of the process $\Xi$ under $\prob^{\x'}$
and $\prob^{\x''}$ are also equal.
\end{lemma}

For the remainder of this section, we will make the following assumption.

\begin{assumption} 
\label{ass:immediate_hit}
The Feller process $X$ is such that
if $X'$ and $X''$ are two independent copies of $X$, then,
for all $t_0 > 0$ and $x' \in E$,
\[ 
\lim_{x'' \to x'} \prob^{x',x''} \big \{ X_t' = X_t'' \text{ for some } t \in [0, t_0] \big \} = 1.
\]
\end{assumption}

\begin{proposition} \label{p:inv_dense_set}
Let $\x', \x'' \in E^N$ be such that the sets
$\{x_i' : i \in [N]\}$ and $\{x_i'' : i \in [N]\}$ have the same
closure. Then, the process $\Xi$ has the same distribution under $\prob^{\x'}$
and $\prob^{{\x}''}$.
\end{proposition}

\begin{proof}
We will consider the case where $E$
is compact.  The non-compact case is essentially the same, and we leave
the details to the reader.

We need to show for any finite set of
times $0 < t_1 < \ldots < t_k$ that the distribution
of $(\Xi_{t_1}, \ldots, \Xi_{t_k})$ is the same under $\prob^{\x'}$
and $\prob^{x''}$.  

We may suppose without loss of generality that 
$x_1', x_2', \ldots$ (resp. $x_1'', x_2'', \ldots$) are distinct.  

Fix $n \in [N]$ and $\delta>0$.
Given $\eps > 0$ that will be specified later, choose
$y_1'', y_2'', \ldots, y_n'' \in \{x_i'' : i \in [N]\}$ 
such that $d(x_i', y_i'') \le \eps$ for $1 \le i \le n$. 
Let $\bfeta'$ (resp. $\bfeta''$) be an $E^n$-valued process with
coordinates that are independent copies of $X$ started at
$(x_1', \ldots, x_n')$ (resp. $(y_1'', y_2'', \ldots, y_n'')$).

By the Feller property, there is a time
$0 < t_0 \le t_1$ that depends on
$x_1', \ldots, x_n'$
such that for all $\eps$ sufficiently small 
\[
\prob\{\eta_i''(t) = \eta_j''(t) \text{ for some $1 \le i \ne j \le n$ and $0 < t \le t_0$}\} \le \frac{\delta}{2}.
\]
By our standing Assumption~\ref{ass:immediate_hit},
if we take $\eps$ sufficiently small, then
\[
\prob\{\eta_i'(t) \ne \eta_i''(t) \text{ for all  $0 < t \le t_0$}\} \le \frac{\delta}{2n}, \quad 1 \le i \le n.
\]

Write $\Xi'$ (resp. $\Xi''$, $\hat \Xi$, $\check \Xi$) for the set-valued processes
constructed from $\bfeta'$ (resp. $\bfeta''$, $(\bfeta',\bfeta'')$, $(\bfeta'',\bfeta')$)
in the same manner that $\Xi$ is constructed from $\bxi$.  We have
\[
\prob\{\check \Xi_t = \Xi_t'' \text{ for all $t \ge t_0$}\} \ge 1 - \delta,
\]
\[ 
\Xi_t' \subseteq \hat \Xi_t, \quad \text{for all $t \ge 0$},
\]
and, by Lemma~\ref{l:inv_dist},
\[
\hat \Xi
\stackrel{d}{=} 
\check \Xi.
\]

For each $z \in E$, define a continuous function
$\phi_z : \kcal \to \mathbb{R}_+$ by
\[
\phi_z(K) := \inf\{d(z,w) : w \in K\}.
\]
Note that $K' \subseteq K''$ implies that $\phi_z(K') \ge \phi_z(K'')$ for any $z \in E$.
It follows that for points $z_{\ell p} \in E$, $1 \le p \le q_\ell$,
$1 \le \ell \le k$,
\[
\begin{split}
\E
\left[
\prod_{\ell=1}^k \prod_{p=1}^{q_\ell}
\phi_{z_{\ell p}}(\Xi_{t_\ell}')
\right]
& \ge
\E
\left[
\prod_{\ell=1}^k \prod_{p=1}^{q_\ell}
\phi_{z_{\ell p}}(\hat \Xi_{t_\ell})
\right] \\
& =
\E
\left[
\prod_{\ell=1}^k \prod_{p=1}^{q_\ell}
\phi_{z_{\ell p}}(\check \Xi_{t_\ell})
\right] \\
& \ge
\E
\left[
\prod_{\ell=1}^k \prod_{p=1}^{q_\ell}
\phi_{z_{\ell p}}(\Xi_{t_\ell}'')
\right]
-
\delta \left(\sup\{d(z,w) : z,w \in E\}\right)^{\sum_\ell q_\ell} \\
\end{split}
\]
Observe that
\[
\E^{\x'}
\left[
\prod_{\ell=1}^k \prod_{p=1}^{q_\ell}
\phi_{z_{\ell p}}(\Xi_{t_\ell})
\right]
=
\lim_{n \to \infty}
\E^{\x'}
\left[
\prod_{\ell=1}^k \prod_{p=1}^{q_\ell}
\phi_{z_{\ell p}}(\Xi_{t_\ell}')
\right]
\]
and
\[
\E
\left[
\prod_{\ell=1}^k \prod_{p=1}^{q_\ell}
\phi_{z_{\ell p}}(\Xi_{t_\ell}'')
\right]
\ge
\E^{\x''}
\left[
\prod_{\ell=1}^k \prod_{p=1}^{q_\ell}
\phi_{z_{\ell p}}(\Xi_{t_\ell})
\right].
\]
Since $\delta$ is arbitrary, 
\[
\E^{\x'}
\left[
\prod_{\ell=1}^k \prod_{p=1}^{q_\ell}
\phi_{z_{\ell p}}(\Xi_{t_\ell})
\right]
\ge
\E^{\x''}
\left[
\prod_{\ell=1}^k \prod_{p=1}^{q_\ell}
\phi_{z_{\ell p}}(\Xi_{t_\ell})
\right].
\]
Moreover, we see from interchanging the roles of $\x'$
and $\x''$ that the last inequality is actually an equality.

It remains to observe from the Stone-Weierstrass theorem that
the algebra of continuous functions generated by the constants
and the set $\{\phi_z : z \in E\}$ is uniformly dense in the
space of continuous functions on $E$.
\end{proof}

With Proposition~\ref{p:inv_dense_set} in hand, it makes
sense to talk about the distribution of the process $\Xi$
for a given initial state $\Xi_0$.
The following result follows immediately from Dynkin's criterion
for a function of Markov process to be also Markov.

\begin{corollary}
The process $(\Xi_t)_{t \ge 0} $ is strong Markov with
respect to the filtration $(\fcal_t)_{t\ge 0}$.
\end{corollary}

\subsection{Coalescing marked particles}\label{subsec:coupling}

Starting with the Feller Markov process $X$ on $E$, we can take
another locally compact, second-countable, Hausdorff
{\em mark} space $M$ and build a Feller Markov process $\hat X$
with state space $\hat E = E \times M$ by taking the distribution
of $(\hat X_t)_{t \ge 0}$ when $\hat X_0 = (x,m)$ to be that
of $((X_t,m))_{t \ge 0}$ when $X_0 = m$.  That is,
the $E$-valued component of $\hat X$ evolves in the same manner as $X$,
while the $M$-valued component stays at its initial value.  

Given a ranking $\sigma$ of $[N]$, we can define a collision rule
$\hat \Lambda_\sigma$ for $\hat E$-valued paths in the same
way that we defined the collision rule 
$\Lambda_\sigma$ for $\E$-valued paths.  Note that if 
$\bxi = (\xi_1, \xi_2, \ldots)$ is a vector of $E$-valued paths
and we define a vector 
$\hat \bxi = (\hat \xi_1, \hat \xi_2, \ldots)$ of 
$\hat E$-valued paths by $\hat \xi_i(t) = (\xi_i(t),m_i)$ for
$m_1, m_2, \ldots \in M$, then it is {\bf not} the case that
vector of $E$-valued components of 
$\hat \bzeta := \hat \Lambda_\sigma \hat \bxi$
is always equal to $\bzeta := \Lambda_\sigma \bxi$: in order for the
$E$-valued components of two particles to coalesce from some time
onwards, the corresponding unchanging $M$-valued marks have to agree.



Because particles can coalesce in the $\bzeta$ system that are unable to
coalesce in the $\hat \bzeta$ system, it might seem at first glance
that for $N=n$ we have
\[ 
\{\zeta_i(t): 1 \le i \le n \} 
\subseteq 
\{z_i : \hat \zeta_i(t) = (z_i, m_i), \,  1 \le i \le n \}
\]
for all $t \ge 0$.  However, it is not too difficult to construct
examples where preventing particles from coalescing at an early
stage of the evolution leaves several particles around at a later
stage in the correct locations and with the correct marks 
to lead to an excess of coalescences over what occurs in the
unmarked system.  Nonetheless, an ordering of this sort
holds in the sense of stochastic domination rather than pointwise.
More precisely, the following claim holds.

\bigskip
\noindent
\textbf{Claim.} 
Given $\bxi = (\xi_i)_{i=1}^n$ and marks $(m_i)_{i=1}^n$, we can
construct $\tilde \bzeta = (\tilde \bzeta_i)_{i=1}^n$ that has
the same distribution as $\hat \bzeta$ and is such that, almost surely,
\[ \{\zeta_i(t): 1 \le i \le n \} 
\subseteq 
\{z_i : \tilde \zeta_i(t) = (z_i, m_i), \,  1 \le i \le n \}\]
 for all $t \ge 0$.



Before we present the formal construction of $\tilde \bzeta$, we give the
following verbal description which may help the reader.  We have hitherto
defined a ranking for an $n$-particle system  to be a bijection from
$[n]$ to $[n]$, but it will be convenient to modify this definition
and now take a ranking of $n$ particles to be an injection from $[n]$
to $\mathbb{N}$ (the previous definition can be thought of as
the special case of this one where the image of the injection is $[n]$).

\begin{itemize}
\item[(i)]
Imagine that at any given time each particle can be one of three types:  
{\em active}, {\em injured}, or {\em dead}. 
\begin{itemize}
\item[(a)]
All particles are initially active. 
\item[(b)]
An active particle can remain active or become either injured or dead.
\item[(c)]  
An injured particle can remain injured or become dead. 
\item[(d)] A dead particle remains dead.
\end{itemize}
\item[(ii)] Suppose that an active particle  collides with another active particle. 
\begin{itemize} 
\item[(a)]
The particle with smaller rank (at the time of the collision) remains active. 
\item[(b)] 
If the two colliding particles have same mark, then  the particle with the higher rank becomes dead and follows the path of the other particle thereafter. There is no change in the rankings of any particle. 
\item[(c)]
If the colliding particles have different marks,  then the particle with the higher rank becomes injured.  The higher rank particle continues to follow its own path.  Its ranking and the rankings of all the particles that have already coalesced with it are increased by $n$.  
The rankings of all other particles remain unchanged.
\end{itemize} 
\item[(iii)] 
Suppose that an injured particle collides with an active particle  with the same mark. Then, the injured particle becomes dead and follows the path of the active particle thereafter.  The rankings of all particles are unchanged.
\item[(iv)]
Suppose that two injured particles sharing the same mark collide.   Then,   the particle with the higher rank becomes dead and follows the path of the particle with lower rank thereafter.  The rankings of all particles are unchanged.
\item[(v)] If there is a collision between any pair of particles not described above,  then both of the colliding  particles continue to
follow their own paths and there is no change in the ranking.
 \end{itemize}

We now give a more formal description of the above construction.  
Let $\upsilon_0 = 0 < \upsilon_1< \upsilon_2< \cdots$ be the successive collision times for the process $\bzeta = \Lambda_\sigma \bxi$, that is,
\begin{equation}\label{eq:succ_collision}
\upsilon_{i+1} := \inf\{t > \upsilon_i: \zeta_j(t) = \zeta_k(t), \, \zeta_j(\upsilon_i) \ne \zeta_k(\upsilon_i) \}.
\end{equation}
To build our coalescing system, we first define 
an $(\fcal_t)_{t \ge0} $-adapted ranking-valued process 
$(\sigma_t)_{t\ge 0}$ which starts from $\sigma_0$ at time $t=0$
and is constant on each interval $[ \upsilon_i, \upsilon_{i+1})$. 
For $i \ge 1$ and $k \in [n]$, set
 \[ \sigma_{\upsilon_i}(k) : = 
 \begin{cases}
  \sigma_{\upsilon_i-}(k)+n, & \text{ if } \exists j \in [n], \text{ such that } \zeta_j(t) = \zeta_k(t), \\
  & \quad  \zeta_j(\upsilon_{i-1}) \ne \zeta_k(\upsilon_{i-1}), \, m_j \ne m_k, \, \sigma(j) < \sigma(k),\\
  \sigma_{\upsilon_i-}(k), & \text{ otherwise.} 
 \end{cases}
  \]
For $\upsilon_{i} \le  t < \upsilon_{i+1}$, set
\[ 
\tilde \bzeta(t) := \hat \Lambda_{\sigma_{\upsilon_i}} \Big(\hat \theta_{\upsilon_i - \upsilon_{i-1}} \circ \hat \Lambda_{\sigma_{\upsilon_{i-1}} } \ldots \big(\hat \theta_{\upsilon_2- \upsilon_1} \circ \hat \Lambda_{\sigma_{\upsilon_1}} ( \hat \theta_{\upsilon_1- \upsilon_0} \circ \hat \Lambda_{\sigma_{\upsilon_0}} \hat \bxi) \big) \Big) (t-\upsilon_i),  \]
where the collision operators $\hat \Lambda_\sigma$ are defined as before
(the definition continues to make sense with our more general
notion of ranking) 
and $(\hat \theta_t)_{t \ge 0}$ is the family of shift operators on the space of $\hat E$-valued paths. 

The following observations prove the claim.
 \begin{enumerate}
 \item The rank of an injured particle is always higher than that of an active particle.
 
 \item During the evolution of the process, the relative ranking of the active particles is unchanged. Thus, the set of $E$-valued components
 of the locations of the active particles present at time $t$ evolves as the set-valued coalescing  process corresponding to 
 $\bzeta = \Lambda_{\sigma_0} \bxi$.

 \item The vector-valued process $\tilde \bzeta$ has the same distribution
 as the coalescing process $\hat \bzeta$. 
Indeed, if we define 
the successive collision times $\hat \upsilon_i$, $i \ge 0$,  
for the process $\hat \bzeta$ by analogy with \eqref{eq:succ_collision}, then 
it follows by induction and 
the strong Markov property of  the process $\hat \bxi $ with respect to filtration $(\fcal_t)_{t \ge 0}$
that the distribution of the process 
$( \hat \upsilon_i \wedge t, \hat \bzeta_{ \hat \upsilon_i \wedge t })_{ t \ge 0}$ does not depend on the  ranking process $(\sigma_{\hat \upsilon_i \wedge t} )_{ t \ge 0}$ when $(\sigma_t)_{ t \ge 0} $ is $(\fcal_t)_{ t \ge 0}$-adapted.

 \end{enumerate}

\section{Processes on compact spaces}
\label{S:compact_space}

The conditions of the following theorem are shown in \cite{MR2008600}
to hold for symmetric processes with suitable Dirichlet forms
on $d$-sets in $\mathbb{R}^n$, $0 < d \le n$.  They certainly hold
for the symmetric stable processes on the circle, with $d=1$
and $1 < \alpha < 2$ the stable index.  The latter
processes are, in any case, instances
of the processes considered in \cite{MR2008600}, where other examples
such as stable subordinations of suitable diffusions on fractals
are also discussed.

\begin{theorem}[Instantaneous Coalescence]\label{thm_main1_d-set}
Suppose that $(E,r)$ is a compact metric space equipped with a Borel
probability measure $\mu$ such that
\[
C_1 \epsilon^d \le \mu(B(x,\epsilon)) \le C_2 \epsilon^d, \quad x \in E, \, 
0 < \epsilon \le 1,
\]
for constants $0 < C_1 < C_2$, 
where $B(x,\epsilon)$ is the open ball of radius $\epsilon$ centered at $x$.
Consider a Feller Markov process $X$
with state space $E$ that has jointly continuous transition densities
$(t,x,y) \mapsto p(t,x,y)$ with respect to $\mu$.  Assume
that $X$ is symmetric with respect to $\mu$, so that
$p(t,x,y) = p(t,y,x)$.  Assume further that for some $\alpha > d$
we have bounds of the form
\[
c_1 
\left\{
t^{- d/\alpha}
\wedge
\frac{t}
{r(x, y)^{d+\alpha}}
\right\}
\le
p(t,x,y)
\le
c_2 
\left\{
t^{- d/\alpha}
\wedge
\frac{t}
{r(x, y)^{d+\alpha}}
\right\}, \quad 0 < t \le 1,
\]
for suitable constants $0 < c_1 < c_2$. 
Let $\Xi$ be the corresponding set-valued coalescing system. 
Then, almost surely, $\Xi_t$ is a finite set for all $t>0$.
\end{theorem}

\begin{proof}
We will verify the bounds 
\eqref{hitting_lower_bound} and \eqref{Lebesgue_number},
so that we can apply the argument in the Introduction.

Let $X'$ and $X''$ be two independent copies of $X$
started from $x'$ and $x''$, respectively.  We want a lower
bound on the probability 
\[
\mathbb{P}\{\exists 0 \le s \le t : X_s' = X_s''\}.
\]
To this end, set 
$W_\epsilon := \int_0^t \mathbf{1}\{r(X_s',X_s'') \le \epsilon\} \, ds$
and note by the Cauchy--Schwarz inequality that
\begin{equation}
\label{second_moment_method}
\begin{split}
& \mathbb{P}\{\exists 0 \le s \le t : X_s' = X_s''\}
= \lim_{\epsilon \downarrow 0} \mathbb{P}\{W_\epsilon > 0\} 
\ge
\liminf_{\epsilon \downarrow 0}
\frac{\mathbb{E}[W_\epsilon]^2}
{\mathbb{E}[W_\epsilon^2]} \\
& \quad =
\frac
{\left[\int_0^t \int_E p(s,x',y) p(s,x'',y) \, \mu(dy) \, ds\right]^2}
{2 \int_0^t \int_s^t \int_E \int_E
p(s,x',y)p(s,x'',y) p(u-s,y,z) p(u-s,y,z) \, \mu(dy) \, \mu(dz) \, du \, ds} \\
& \quad =
\frac
{\left[\int_0^t  p(2s,x',x'')  \, ds\right]^2}
{2 \int_0^t \int_s^t \int_E
p(s,x',y)p(s,x'',y) p(2(u-s),y,y)  \, \mu(dy)  \, du \, ds}. \\
\end{split}
\end{equation}

For $t = \frac{1}{2} r(x', x'')^\alpha$,
the numerator in \eqref{second_moment_method} is bounded below by
\[
\begin{split}
& \left[
c_1 \frac{1}{2} \int_0^{r(x', x'')^\alpha}
v^{-d/\alpha}
\wedge
\frac{v}
{r(x', x'')^{d+\alpha}} \, dv
\right]^2 \\
& \quad =
\frac{c_1^2}{4}
\left[
\int_0^{r(x', x'')^\alpha}
\frac{v}{r(x', x'')^{d+\alpha}} \, dv
\right]^2 \\
& \quad \ge
c_3 r(x', x'')^{2(\alpha -d )} \\
\end{split}
\]
for a suitable constant $c_3$.
For the same value of $t$, the denominator is bounded above by
\[
\begin{split}
& 2 c_2 \int_0^{r(x', x'')^\alpha/2} \int_s^{r(x', x'')^\alpha/2} \int_E
p(s,x',y)p(s,x'',y) (2(u-s))^{-d/\alpha}  \, \mu(dy)  \, du \, ds \\
& \quad =
c_4 \int_0^{r(x', x'')^\alpha/2}
p(2s,x',x'') (r(x', x'')^\alpha/2 - s)^{1 - d/\alpha} \, ds \\
& \quad \le
c_5
\left[
\int_0^{r(x', x'')^\alpha} 
\frac{v}{r(x', x'')^{d+\alpha}} (r(x', x'')^\alpha/2 - v/2)^{1 - d/\alpha} \, dv
\right] \\
& \quad \le
c_6 r(x', x'')^{2(\alpha -d )} \\
\end{split}
\]
for suitable constants $c_4, c_5, c_6$.  Thus,
\[
\mathbb{P}\left\{\exists 0 \le s \le \frac{1}{2} r(x', x'')^\alpha : X_s' = X_s''\right\}
\ge
p := \frac{c_3}{c_6} > 0
\]
and \eqref{hitting_lower_bound} holds.

Turning to \eqref{Lebesgue_number},
note that if $n$ points of $E$ are such that each point is distance at least
$\epsilon$ from any other, then $n C_1 (\frac{\epsilon}{2})^d \le \mu(E) = 1$.
Hence, in any set with more than $n$ points there must
be at least two points
at distance at most $2 C_1^{-1/d} n^{-1/d}$ apart.

We can therefore apply the argument in the Introduction with 
$\kappa = \frac{1}{d}$, because $\alpha \kappa = \frac{\alpha}{d} > 1$
by assumption.  However, there is one small technical point
that needs to be taken care of.  The construction of the
set-valued coalescing process $\Xi$ was carried out under the assumption
that Assumption~\ref{ass:immediate_hit} holds, and we need to verify that
this is the case.
It follows from the continuity of the transition densities 
and the Markov property that
$\mathbb{P}\{\exists \delta \le s \le t : X_s' = X_s''\}$ is jointly
continuous in the starting points $x'$ and $x''$ for $0 < \delta \le t$.
By the Blumenthal zero-one law it
 therefore suffices to show for $x' = x'' = x$ that
\begin{equation}
\label{input_to_Blumenthal}
0
<
\inf_{t > 0} \mathbb{P}\{\exists 0 < s \le t : X_s' = X_s''\}
= 
\inf_{t>0} \lim_{\delta \downarrow 0} 
\mathbb{P}\{\exists \delta \le s \le t : X_s' = X_s''\}.
\end{equation}
The argument that led to \eqref{second_moment_method} shows
the limit in rightmost term of \eqref{input_to_Blumenthal} is bounded below by
\[
\begin{split}
& \frac
{\left[\int_0^t \int_E p(s,x,y) p(s,x,y) \, \mu(dy) \, ds\right]^2}
{2 \int_0^t \int_s^t \int_E \int_E
p(s,x,y)p(s,x,y) p(u-s,y,z) p(u-s,y,z) \, \mu(dy) \, \mu(dz) \, du \, ds} \\
& \quad =
\frac
{\left[\int_0^t  p(2 s, x, x)  \, ds\right]^2}
{2 \int_0^t \int_s^t \int_E 
p(s,x,y)p(s,x,y) p(2(u-s),y,y) \, \mu(dy)  \, du \, ds}. \\
\end{split}
\]
For small $t > 0$, the numerator is bounded below by
\[
\left[c_1 2^{d/\alpha} \int_0^t s^{-d/\alpha} \, ds \right]^2 = c_7 t^{2(1 - d/\alpha)}
\]
for a suitable (positive) constant $c_7$.  
Similarly, the denominator is bounded above by
\[
\begin{split}
c_8 \int_0^t \int_s^t s^{-d/\alpha} (u-s)^{-d/\alpha} \, du \, ds
& =
c_8  \frac{4^{d/\alpha-1} \sqrt{\pi } \Gamma
   (1-d/\alpha) t^{2(1-d/\alpha)} }{(1-d/\alpha) \Gamma \left(\frac{3}{2}-d/\alpha\right)} \\
& =
c_9 t^{2(1 - d/\alpha)} \\
\end{split}
\]
for suitable (finite) constants $c_8$ and $c_9$.  Therefore, the rightmost term
of \eqref{input_to_Blumenthal} is bounded below by $c_7/c_9 > 0$, as required.
\end{proof}

\section{Brownian motion on the Sierpinski gasket}
\label{S:gasket}

\subsection{Definition and properties of the gasket}

Let 
\[
G_0 := \{ (0,0), (1, 0), (1/2, \sqrt3 /2) \}
\]
be the vertices of the unit triangle in $\mathbb R^2$ and denote by $H_0$ the closed convex hull of $G_0$. The {\em Sierpinski gasket}, which we also call the {\em finite gasket}, is
a fractal subset of the plane that can be constructed via the following Cantor-like cut-out procedure. Let $\{ b_0, b_1, b_2\}$ be the midpoints of three sides of $H_0$ and let $A$ be the interior of the triangle with vertices $\{b_0, b_1, b_2\}$. Define $H_1 := H_0 \setminus A$ so that $H_1$ is the union of $3$ closed upward facing triangles of side length $2^{-1}$. Now repeat this operation on each of the smaller triangles to obtain a set $H_2$, consisting of $9$ upward facing closed triangles, each of side $2^{-2}$. Continuing this fashion, we have a decreasing sequence of closed non-empty sets $\{H_n\}_{n=0}^\infty$ and we define the Sierpinski gasket as 
\[ G := \bigcap_{n=0}^\infty H_n.\]

We call each of the $3^n$ triangles of side $2^{-n}$ that make up $H_n$ an {\em $n$-triangle} of $G$.  
Denote by $\tcal_n$ the collection of all $n$-triangles of $G$.
 Let $\vcal_n$ be the set of vertices of the $n$-triangles.

We call the unbounded set  
\[ \wt G := \bigcup_{n=0}^\infty 2^n G \]
the {\em infinite gasket} (where, as usual, we write
$c B := \{c x : x \in B\}$ for $c  \in \mathbb{R}$ and $B \subseteq \mathbb{R}^2$) . 
The concept of $n$-triangle, where
$n$ may now be a negative integer, extends in
the obvious way to the infinite gasket. Denote the set of all $n$-triangles of $\wt G$ by $\wt \tcal_n$. Let $\wt \vcal_n$ be the vertices of $\wt \tcal_n$.

Given a pathwise connected subset 
$A \in \mathbb R^2$, let $\rho_A$ be the {\em shortest-path metric} on $A$
given by
\[ \rho_A(x, y) := \inf\{|\gamma| : \gamma \text{ is a path between } x \text{ and } y \text{ and } \gamma \subseteq A \},\]
where $|\gamma|$ denote the length (that is, the $1$-dimensional
Hausdorff measure) of $\gamma$.
For the finite gasket $G$, $\rho_G$ is comparable to the usual Euclidean metric $| \cdot |$ (see, for example, \cite[Lemma 2.12]{barlow95}) with the relation, 
\[ |x- y| \le \rho_G(x,y) \le c|x-y|, \quad \forall x, y \in G,\]
for a suitable constant $1 < c < \infty$.
It is obvious that the same is also true for the metric 
$\rho_{\wt G}$ on the infinite gasket.

Let $\mu$ denote the $d_f$-dimensional Hausdorff measure
on $\wt G$ where $d_f := \log 3/ \log 2$
is the {\em fractal} or {\em mass dimension} of the gasket.
For the finite gasket $G$ we have $0 < \mu(G) < \infty$
and, with a slight abuse of notation,
we will also use the notation $\mu$ to denote
the restriction of this measure to $G$.
Moreover, we have the following estimate on the volume growth of  $\mu$
\begin{equation} \label{eq:volgrowth}
C' r^{d_f} \le   \mu ( B(x, r) ) \le C r^{d_f} \quad \text{for } x \in \wt G, \; 0 < r < 1,
\end{equation}
where $B(x,r) \subseteq \wt G$ is the open ball with center $x$ and radius $r$ 
in the Euclidean metric and $C, C'$  are suitable
constants (see \cite{barlow88}).

\subsection{Brownian motions}

We  construct a graph $G_n$ (respectively,
$\wt G_n$) embedded in the plane 
with vertices $\vcal_n$ (resp. $\wt \vcal_n$)
by adding edges between pairs of vertices that are 
distance $2^{-n}$ apart from each other. Let $X^n$ (resp. $\wt X^n$) be the natural random walk on $G_n$ (resp. $\wt G_n$);
that is, the discrete time Markov chain that at each step
chooses uniformly at random from one of the neighbors
of the current state. It is known (see \cite{barlow88, barlow95}) that  
the sequence $(X^n_{\lfloor 5^nt \rfloor})_{ t \ge 0}$ (resp. $(\wt X^n_{\lfloor 5^nt \rfloor})_{t \ge 0}$) converges in distribution as $n \rightarrow \infty$
 to a limiting process $(X_t)_{t \ge 0}$ (resp. $(\wt X_t)_{t \ge 0}$) that is a $G$-valued (resp. $\wt G$-valued) strong Markov process (indeed, a Feller process)
 with continuous sample paths. The  processes $X$ and $\wt X$ are called, for obvious reasons,  the Brownian motion  on the finite and infinite gaskets, respectively. The Brownian motion on the {\em infinite} gasket has the following scaling property:
\begin{equation}\label{e:bmscaling}
\text{
$(2 \wt X_t)_{t \ge 0}$ 
under  $\prob^x$ has same law as 
$(\wt X_{5t})_{t \ge 0}$
under
$\prob^{2x}$.} 
\end{equation}

The process  $\wt X$ has a family $\wt p(t, x, y$), $x, y \in \G$, $t > 0$, of transition densities
with respect to  the measure $\mu$ that is jointly continuous on $(0, \infty) \times \wt G \times \wt G$. Moreover, 
$\wt p(t, x, y) = \wt p(t, y, x)$ for all $x, y \in \wt G$
and $t> 0$, so that the process $\wt X$ is symmetric with respect to $\mu$.

Let $ d_w  := \log 5 / \log 2$ denote the {\em walk dimension}
of the gasket.  The following crucial ``heat kernel bound'' is established 
in \cite{barlow88}
\begin{equation} 
\label{eq:densitybound}
\begin{split}
& c'_1 t^{ -d_f/ d_w} \exp \left( - c'_2\left( \frac{|x-y|^{d_w}}{t}\right)^{ 1/ (  d_w-1)} \right) \\
& \quad \le \p(t, x, y) \\
& \quad \le c_1 t^{ -d_f/ d_w} \exp \left( - c_2\left( \frac{|x-y|^{d_w}}{t}\right)^{ 1/ (  d_w-1)} \right), \quad \forall x, y \in \G, \,  t >0.
\end{split}
\end{equation}

Because the infinite gasket $\wt G$ and the associated Brownian motion $\wt X$ both have
re-scaling invariances that $G$ and $X$ do not, it will
be convenient to work with $\wt X$ and then use
the following observation to transfer
our results to $X$.  

\begin{lemma}[Folding lemma]
\label{infinite_to_finite_gasket}
There exists a continuous mapping $\psi: \G \to G$ such that
$\psi$ restricted to $G$ is the identity,
$\psi$ restricted to any $0$-triangle is an isometry,
and $|\psi(x) - \psi(y)| \le |x - y|$ for arbitrary $x,y \in \G$.
Moreover, if
the $\G$-valued process $\wt X$ is started at an arbitrary $x \in \G$,
then the $G$-valued process $\psi \circ \wt X$ 
has the same distribution the process $X$ started at $\psi(x)$.
\end{lemma}

\begin{proof} 
Let $L$ be the subset of
the plane formed by the set of points of the
form $n_1 (1, 0) + n_2 (1/2, \sqrt3 /2)$,
where $n_1, n_2$ are non-negative integers,
and the line segments that join such points that 
are distance $1$ apart.  
It is easy to see that there
is a unique labeling of the vertices of $L$
by $\{1,\omega,\omega^2\}$ 
that has the following properties.
\begin{itemize}
\item Label $(0,0)$ with $1$.
\item If vertex $v$ is labeled $\mathfrak a \in \{1,\omega,\omega^2\}$, then the vertex 
$v + (1, 0)$ are labeled with $\mathfrak a \go$.
\item 
If we think of the labels as referring to 
elements of the cyclic group of order $3$, then
if vertex $v$ is labeled $\mathfrak a \in \{1,\omega,\omega^2\}$, 
then vertex $v + (1/2, \sqrt3 /2)$ is labeled with $\mathfrak a \go^2$.
\end{itemize}
Indeed, the label of the
vertex $n_1 (1, 0) + n_2 (1/2, \sqrt3 /2)$ is 
$\omega^{n_1 + 2 n_2}$.

Given a vertex $v \in L$,  let $ \iota(v)$ be the
unique vertex in $ \{ (0,0), (1, 0),  (1/2, \sqrt3 /2)\}$ that
has the same label as $v$. 
If the vertices $v_1, v_2, v_3 \in L$ are the vertices
of a triangle with side length $1$, then $\iota(v_1), \iota(v_2), \iota(v_3)$ are all distinct.

With the above preparation, let us now define the map $\psi$. 
Given $x \in \G$, let  $\gD \in \wt \tcal_0$ be a triangle with vertices $v_1, v_2, v_3$ that contains $x$ (if $x$ belongs
to $\wt \vcal_n$, then there may be more than one such triangle,
but the choice will not matter).
We may write $x$  as a
unique convex combination of the vertices $v_1, v_2, v_3$, 
\[ x = \gl_1 v_1+ \gl_2 v_2+ \gl_3 v_3, \quad \sum_{i=1}^3 \gl_i =1, \gl_i \ge 0.
\]
The triple $(\gl_1, \gl_2, \gl_3)$ is the vector of {\em barycentric}
coordinates of $x$.
We define $\psi(x)$ by
\[ \psi(x) := \gl_1 \iota(v_1)+ \gl_2 \iota(v_2)+ \gl_3 \iota(v_3). \]
It is clear that $\psi: \G \to G$ is well-defined and has the stated properties.

Recall that $\wt X^{(n)}$ be the natural
 random walk on $\wt G_n$. It can be verified easily that the projected process 
$\psi \circ \wt X^{(n)}$ is the natural 
random walk on $G_n$.
The result follows by taking the limit as 
$n \to \infty$ and using the continuity of $\psi$.
\end{proof}

\begin{lemma}[Maximal inequality]\label{l:supbound}
(a) Let $\wt X^i$, $ 1 \le i \le n$, be
$n$ independent Brownian motions on the infinite gasket $\wt G$ starting from the initial states $x^i$, $1 \le i \le n$. For any $t >0$,
\[  \prob \left \{\sup_{0 \le s \le t} |\wt X^i_s - x^i| > r, \text{ for some } 1 \le i \le n \right \} \le 2nc_1 \exp\Big ( - c_2 (  r^{d_w}/t )^{ 1/ (d_w - 1)}\Big ),  \]
where $c_1, c_2>0$ are constants and $d_w = \log 5/ \log 2$ is the walk dimension of the gasket.

\noindent
(b) The same estimate holds for the case of $n$ independent Brownian motions $X^i$, $1 \le i \le n$, on the finite gasket $G$ starting from the initial states
$x^i$, $1 \le i \le n$.
\end{lemma}

\begin{proof} (a) Let  $\wt X = (\wt X_t)_{t\ge 0}$ be a Brownian motion on $\wt G$. Then for $x\in \wt G$, $t> 0$, and $r> 0$, 
\[ \begin{split}
\prob^x \left \{\sup_{0 \le s \le t} |\wt X_s - x| > r \right\} 
& \le \prob^x \{ |\wt X_t - x| > r/2\}\\ 
& \quad + \prob^x\big \{ |\wt X_t - x| \le r/2, \sup_{0 \le s \le t} |\wt X_s - x| > r\big\}. \\
\end{split}
\]
Writing $S := \inf\{s> 0:  |\wt X_s - x| > r \}$, the second term above equals 
\[
\E^x \left[ 1_{\{S < t\}} \prob^{\wt X_S} \{ |\wt X_{t-S} - x| \le r/2\} \right] 
\le \sup_{ y \in \partial B(x, r)} \sup_{s \le t} \prob^y \{ |\wt X_{t-s} - y| > r/2\},   
\]
where $\partial B(x, r)$ is the boundary of $B(x,r)$
so that 
\begin{align*}
\prob^x \left\{\sup_{0 \le s \le t} |\wt X_s - x| > r \right\} &\le 2 \sup_{ y \in \wt G} \sup_{ s \le t} \prob^y\{|\wt X_s  - y| > r/2\}\\
\label{supbound}
&\le 2c_1 \exp\Big ( - c_2 (  r^{d_w}/t )^{ 1/ (d_w - 1)}\Big ),
\end{align*}
where the last estimate is taken from \cite[Theorem 2.23(e)]{barlow95}.  The lemma now follows by a union bound.

(b) This is immediate from part (a) and Lemma~\ref{infinite_to_finite_gasket}.
\end{proof}

\subsection{Collision time estimates}

We first show that two independent copies of $\wt X$ collide with positive probability.

\begin{proposition}
\label{prop:diag_nonpolar}
Let $\wt X'$ and $\wt X''$ be two independent copies of $\wt X$.
Then,
\[
\prob^{(x',x'')}\{\exists t > 0 : \wt X_t' = \wt X_t''\} > 0
\]
for all $(x',x'') \in \G \times \G$.
\end{proposition}  

\begin{proof}
Note that  ${\bf \wt X} = (\wt X' , \wt X'')$  is a Feller process on the locally compact separable metric space 
$\wt G \times \wt G$ that  is symmetric
with respect to the Radon measure $\mu \otimes \mu$ and has transition
densities $\wt p(t, x', y') \times \wt p(t, x'', y'')$.  The corresponding $\alpha$-potential density is
\[ u_\alpha(\x, \y) := \int_0^\infty e^{-\alpha t} \wt p(t, x_1, y_1) \times \wt p(t, x_2, y_2) \, dt
\quad \text{for } \alpha > 0, \]
where $\x=(x_1, x_2)$ and $\y = (y_1, y_2)$.
A standard potential theoretic result says that a  compact set
$B \subseteq \wt G \times \wt G$ is non-polar if there exists a non-zero finite measure $\nu$ that is supported on B and has finite energy,
that is,
\[ \int \int u^\alpha(\x,\y)\,  \nu(d\x) \, \nu(d\y) < \infty.\]

Take $B = \{ (x', x'') \in G \times  G: x' = x''\}$ and 
$\nu$ to be the `lifting' of the Hausdorff measure $\mu$
on the finite gasket onto $B$. We want to show that
\[ 
\int_{G} \int_{G}  \int_0^{\infty} 
e^{-\alpha t} \p^2(t, x, y) \, dt \, \mu(dx) \mu(dy) < \infty.
\]
It will be enough to show that 
\[ \int_{G} \int_{ G}  \int_0^{\infty} 
 \p^2(t, x, y) \, dt \, \mu(dx) \, \mu(dy) < \infty.\]

It follows from the transition density
estimate \eqref{eq:densitybound}
and a straightforward integration that
\[ \int_0^{\infty} \p^2(t, x,y) \, dt \le C |x-y|^{ - \gamma }
\]
for some constant $C$, where $ \gamma := 2 d_f - d_w $.   Thus,
\[
\begin{split}
& \int_{G} \int_{G}  \int_0^{\infty} \p^2(t, x,y)\, dt \, \mu(dx) \, \mu(dy) \\
& \quad \le C  \int_{G} \int_{G}  |x-y|^{ - \gamma } \, \mu(dx) \, \mu(dy) \\
 & \quad \le C  \int_{G} \int_0^{\infty} \mu\{ x \in G:  |x-y|^{ - \gamma} > s \} \, ds \, \mu(dy)\\
  & \quad \le C  \int_G \int_0^{\infty} \mu \{ x \in G:  |x-y| < s^{ - 1/ \gamma} \}\,  ds \, \mu(dy)\\
    & \quad \le C  + C\int_{G} \int_1^{\infty} \mu \{ x \in G:  |x-y| < s^{ - 1/ \gamma} \} \, ds  \, \mu(dy)\\
 & \quad \le C  +  C_1\int_{G} \int_1^{\infty}  s^{ - d_f/ \gamma} \, ds \, \mu(dy) \quad [\text{ By } \eqref{eq:volgrowth}]\\
  & \quad \le C  +  C_2 \int_1^{\infty}  s^{ - d_f/ \gamma}  \, ds. 
\end{split}
\]
It remains to note that 
$ \gamma - d_f
= (2 \log 3 / \log 2 - \log 5 / \log 2) - (\log 3 / \log 2)
= (\log 3 - \log 5)/ \log 2 < 0$,
and so $d_f / \gamma < 1$.

This shows that $\prob^{(x',x'')}\{ {\bf \wt X}\text{ hits the  diagonal} \} > 0$ for
some $(x',x'') \in \wt G \times \wt G$.  Because 
$\p^2(t, x, y) > 0$ for all $x, y \in \wt G$ and $t>0$, we even have
$\prob^{(x',x'')}\{ {\bf \wt X} \text{ hits the diagonal} \} > 0$ for all $(x', x'')\in \wt G \times \wt G$.
\end{proof}

We next establish a uniform lower bound  on the collision probability of a pair of independent Brownian motions on the infinite gasket as long as the distance between their starting points remains bounded. 
 
\begin{theorem}\label{l:lower_bdd_hitting_prob}
There exist constants $ \beta >  0$ and 
$ \underline{p} > 0$ such that if $\wt X'$ and $\wt X''$ are two independent Brownian motions on $\wt G$ starting 
from any two points $x, y $ belonging to the same $n$-triangle of $\wt G$, then 
\[
\prob^{(x, y) }\{\wt X_t'  =   \wt X_t''  \text{ for some }  t  \in (0,  \gb 5^{-n} ) \} \ge \underline{p}.
\]
\end{theorem}

This result will require a certain amount of work, so we first note that it leads easily
to an analogous result for the finite gasket.

\begin{corollary}
\label{cor:lower_bdd_hitting_prob}
If $X'$ and $X''$ are two independent Brownian motions on $G$ starting 
from any two points $x, y $ belonging to the same $n$-triangle of $G$, then 
\[
 \prob^{(x, y) }\{X_t'  =   X_t''  \text{ for some }  t  \in (0,  \gb 5^{-n} )  \} \ge \underline{p},
\]
where $ \beta >  0$ and  $ \underline{p} > 0$ are the constants given in Theorem \ref {l:lower_bdd_hitting_prob}.
\end{corollary}

\begin{proof} 
The proof follows immediately  from Lemma~\ref{infinite_to_finite_gasket},
because if  $\wt X_t' =  \wt X_t''$ for some $t$,
then it is certainly the case that 
$\psi \circ \wt X_t' = \psi \circ \wt X_t''$.
\end{proof}

\begin{definition}[Extended triangles for the infinite gasket]
Recall that $\wt \tcal_n$ is the set of all $n$-triangles of $\G$. Given $ \gD \in \wt \tcal_0$ such that $\gD$ does not have the origin as one its vertices, we define the corresponding {\em extended triangle} $\gD^e \subset \G$ as the interior of the union of the original $0$-triangle $\gD$ with the three  neighboring $1$-triangles in $\G$ which share one vertex  with $\gD$ and are not contained in $\gD$. Note that for the (unique) triangle $\gD$ in $\wt \tcal_n$ having the origin as one of its vertices,  there are  two neighboring $1$-triangles in $\G$ that share one vertex  with it which are not contained in $\gD$. In this case, by $\gD^e$, we mean the interior of the union of $\gD$ and these two triangles.
 \end{definition}

  Fix some $\gD \in \wt \tcal_0$. Let $\wt Z$ be the Brownian motion on $\gD^e$ killed when it exits $\gD^e$. It follows from arguments similar to those on
 \cite[page 590]{doob01}, that $\wt Z$ has transition densities
$\p_K(t, x, y)$, $t>0$, $x, y \in  \gD^e$, 
with respect to the restriction of $\mu$ to $\gD^e$,
and these densities have the following properties:
\begin{itemize}

\item $\p_K(t, x, y) = \p_K(t,y,x)$ for all $t>0$, $x, y \in  \gD^e$.

\item $\p_K(t, x, y) \le  \p(t, x, y)$, for all $t>0$, $x, y \in  \gD^e$.

\item $y \mapsto \p_K(t, x, y)$ is continuous for all $t>0$, $x \in  \gD^e$,
and $x \mapsto \p_K(t, x, y)$ is continuous for all $t>0$, $y \in  \gD^e$.

\end{itemize}

It follows that the process
$\wt Z$ is Feller and symmetric with respect to the measure $\mu$.

\begin{lemma}
\label{l:killed_hitting}
Let $ \wt Z',  \wt Z''$ be two independent copies of
the killed Brownian motion
$\wt Z$. Given any $\gee>0$, there exists $0< \gd < \gee$
such that the set of $(x, y) \in \gD^e \times \gD^e$ for which 
\[ 
\prob^{(x, y) }\{ \wt Z_t'  =  \wt Z_t'' \text{ for some } t \in (\delta, \epsilon ) \} > 0
\]   
 has  positive $ \mu \otimes \mu$ mass.
\end{lemma}

\begin{proof}
An argument similar to that in the proof of
Proposition~\ref{prop:diag_nonpolar} shows that
\[
\prob^{(x_0, y_0)}\{\wt Z_t' = \wt Z_t'' \text{ for some } t > 0\} > 0
\]
for some $(x_0,y_0) \in \gD^e \times \gD^e$.

Thus, for any $\epsilon > 0$, we can partition the interval $(0, \infty)$ into the subintervals $ (0, \epsilon)$, $[i\epsilon, (i+1)\epsilon)$, $i \ge 0$ and use the Markov property to deduce that
there exists a point 
$(x_1, y_1) \in  \gD^e \times  \gD^e$  such that
\begin{equation}\label{eq:nonpolargoodpoint}
 \prob^{(x_1, y_1) }\{ \wt Z_t'  = \wt Z_t''  \text{ for some } t \in (0, \epsilon) \} > 0.
 \end{equation}

By continuity of probability, we can find $0< \eta < \epsilon < \infty$ such that
 \[  \prob^{(x_1, y_1) }\{\wt Z_t'  = \wt Z_t''  \text{ for some } t\in (\eta, \epsilon ) \} > 0. \] 
By the Markov property, 
\[
\begin{split}
0 
& < \prob^{(x_1, y_1) }\{\wt Z_t'  =  \wt Z_t''   \text{ for some } t\in (\eta, \epsilon ) \}  \\
& \quad = 
\int_{\gD^e} \int_{\gD^e}  
 \p_K(\eta/2, x_1, x) \p_K(\eta/2, y_1, y)  \\
& \qquad \times \prob^{(x, y) }\{  \wt Z_t'  =  \wt Z_t'',   \text{ for some } t \in  (\eta/2, \epsilon-\eta/2) \}
 \, \mu(dx) \, \mu(dy). \\ 
\end{split}
\]

Therefore,   the initial points $(x, y) \in \gD^e \times \gD^e$ for which
the probability 
\[ 
\prob^{(x, y) }\{  \wt Z_t'  =  \wt Z_t'' \text{ for some } t \in (\eta/2,  \epsilon-\eta/2) \} 
\]   
is positive form a set with positive $\mu \otimes \mu$ measure. The proof now follows by taking $\delta=\eta/2$.
\end{proof}
 
We record the following result for the reader's ease of reference.

\begin{lemma}[Lemma 3.35 of \cite{barlow95}]  \label{lem:escape_prob}
There exists a constant $c_1 > 1$ such that if $x, y \in \gD^e$, $ r = |x - y|$, then
\[ 
\prob^{x} \{ \wt X_t = y  \ \text{ for some } \ t \in (0, r^{d_w}) \ \text{ and } \  | \wt X_t - x| \le c_1 r \  \text{ for all } \ t \le r^{d_w} \} > 0. 
\]
\end{lemma}

\begin{lemma} \label{lem:hit_open_sets}
There exists a constant $c> 0$ such that for each point
$x \in \gD^e$, each open subset $ U \subset \gD^e$, and each time
$0 < t  \le c$
\[ 
\prob^x\{ \wt Z_t \in  U\} > 0.
\]
In particular, $\p_K(t, x, y) > 0$ for all $x, y \in \gD^e$ and $0< t \le c$.
\end{lemma}
\begin{proof} The following three steps combined with  
the strong Markov property establish the lemma.

\noindent
{ \bf Step 1.}  There exists a constant $c > 0$ such that starting from $x \in \gD^e$, the
unkilled Brownian motion on the infinite gasket
$\wt X$ will stay within $\gD^e$ up to time $c$ with positive probability.

\noindent
{ \bf Step 2.}  
 Fix $y \in U$. For all sufficiently small  $ \eta > 0$,
\[
\prob^y \{\wt X \text{ does not exit $U$ before time } \eta\} > 0.
\]

\noindent
{ \bf Step 3.}  For any $\gd >0$, $z, y \in \gD^e$ 
\[ \prob^z\{ \wt Z \text{ hits $y$  before } \gd \} > 0.\]

Consider Step 1. Note that if  $x \in \G$, then (see  \cite[Equation 3.11]{barlow95}) there exists a constant $c> 0$ such that for the unkilled process $\wt X$, we have, 
\[ \prob^x\{ |\wt X_t -x| \le 1/4 \  \text{ for } t \in [0, c] \} > 0 .\]
But if $x \in \gD^e$, then 
\[ \prob^x\{ \wt X_t  \in \gD^e \  \text{ for } t \in [0, c]\} 
\ge 
\prob^x\{ |\wt X_t -x| \le 1/4 \ \text{ for } t \in [0, c]\}, \]
and the claim follows.

Step 2 is obvious from the right
continuity of the paths of the killed Brownian motion $\wt Z$ at time $0$.

Consider Step 3. Fix $z, y \in \gD^e$ and $0 < \gd \le |z- y|$.  Let $\scal_n$ be the $n$-th approximating graph of $\G$ with the set of vertices  $\vcal_n$. 
 Choose $n$ large enough so that we can find points $z_0$ and $y_0$ in $\vcal_n$ close to $z$ and $y$ respectively so that
 \[ 
 |z - z_0| \le \frac {\gd}{3}, \ \  |y - y_0| \le \frac {\gd}{3}  
 \] 
 and
\[  
B(z, c_1|z - z_0| ) \subseteq \gD^e, \ \ B(y_0, c_1|y - y_0| ) \subseteq \gD^e  \]
where $c_1$ is as in Lemma \ref{lem:escape_prob} and the notation $B(u,r)$ 
denotes the intersection with the infinite gasket $\G$
of the closed ball in the plane of radius $r$ around the point $u$.

The length of a shortest path $\gc$ lying $\scal_n$ between $z_0$ and $y_0$ is the
same as  their distance in the original metric  $\rho_{\wt G}(z_0, y_0)$. 
Moreover, for any two points $p$ and $p'$ on $\gc$, 
the length of the segment of $\gc$ between $p$ and $p'$ is the
same as their distance in the original metric $\rho_{\wt G} (p, p')$.

Thus, we can choose $m+1$ equally spaced points $z_0 , z_1, \ldots, z_m = y_0$ on $\gc$ such that
\[  \rho_{\wt G}(z_{i+1}, z_{i} )  = \frac 1{m} \rho_{\wt G}(z_0, y_0)  \quad \text{ for each } i. \]

Since $\gc$ is compact, $\text{dist}(\gc, \partial \gD^e ) > 0$. Thus we can choose $m$ large so that
\[  B(z_{i}, c_1| z_{i+1} - z_i| ) \subseteq \gD^e  \quad \text{ for each } i. \]

By repeated application of Lemma \ref{lem:escape_prob} and the strong Markov property, we conclude that the probability that $\wt Z$ hits $y$ starting from $z$ before the time
\[
T_m := |z - z_0|^{d_w}  + |y - y_0|^{d_w}   + \sum_{i=0}^{m-1} |z_{i+1} - z_{i} |^{d_w}
\]
is strictly positive. Step 3 follows  immediately  since 
\[
T_m \le \left( \frac \gd 3 \right)^{d_w} + \left( \frac \gd 3 \right)^{d_w} + \mathrm{constant} \times m \times  \frac 1{m^{d_w}} |z_0 - y_0|^{d_w} \le \gd
\]
 for $m$ sufficiently large, because $d_w > 1$. 
 \end{proof}

 \begin{lemma} \label{lem:cont}
 Let $\wt Z'$ and $\wt Z''$ be two independent copies of 
 the killed Brownian motion $\wt Z$. For any $0< \gd < \gb$, 
 the map 
 \[ (x, y) \mapsto  \prob^{(x, y) }\{ \wt Z_t'  =  \wt Z_t''  \ \ \text{ for some } t \in (\gd, \gb) \}\] 
 is continuous on $\gD^e \times \gD^e$.
 \end{lemma}
 
\begin{proof}
We have
\[
\begin{split}
&\prob^{(x, y) }\{ \wt Z_t'  =  \wt Z_t''  \ \ \text{ for some } t \in (\gd, \gb) \} \\
& \quad = 
\int_{\gD^e} \int_{\gD^e}
\p_K(\delta, x, x') \, \p_K(\delta, y, y')  \\
& \qquad \times \prob^{(x', y') }\{ \wt Z_t'  =  \wt Z_t''  \ \ \text{ for some } t \in (0, \gb -\gd) \}  
\, \mu(dx') \mu(dy'), \\
\end{split}
\]
and the result follows from the continuity of $z \mapsto \p_K(\delta,z,z')$
for each $z' \in \gD^e$.
\end{proof}

\begin{proof} [Proof of Theorem \ref{l:lower_bdd_hitting_prob}]

For any $x, y \in \gD$, 
\begin{equation}
\label{eq:positive_everywhere}
\begin{split}
& \prob^{(x, y) }\{ \wt Z_t'  =  \wt Z_t''  \text{ for some } t \in (\gd, \gb) \} \\
& \quad = \int_{\gD^e} \int_{\gD^e}  
  \p_K(\gd/2, x, x') \p_K(\gd/2, y, y') \\
& \qquad \times \prob^{(x', y') }\{ \wt Z_t'  =  \wt Z_t''  \text{ for some }  t \in (\gd/2, \gb  -\gd/2) \} 
 \, \mu(dx') \, \mu(dy') > 0, \\
\end{split}
\end{equation}
by Lemmas \ref{l:killed_hitting}, \ref{lem:hit_open_sets} and \ref{lem:cont}.

Applying Lemma  \ref{lem:cont} and equation \eqref{eq:positive_everywhere}
and the fact that a continuous function achieves its minimum on a compact set,  we have
for any $\gD \in \wt \tcal_0$ that
\[ 
\underline{q}(\gD) := \inf_{x, y \in \gD} \prob^{(x, y) }\{ \wt Z_t'  =  \wt Z_t''  \text{ for some } t \in (0, \gb) \} > 0.   
\]
Note that  for any two $\gD_1, \gD_2 \in \wt \tcal_0$ which do not contain the origin,  there exists a {\em local isometry} between the corresponding extended triangles $\gD_1^e, \gD_2^e$. Since the unkilled Brownian motion $\wt X$ in $\G$ is invariant with respect to local isometries,   
\[ \underline{q}(\gD_1) = \underline{q}(\gD_2).\]
Given two independent copies $\wt X'$ and $\wt X''$ of $\wt X$, set
\[ 
\underline{p} := \inf_{ \gD \in \wt \tcal_0} \inf_{x, y \in \gD} \prob^{(x, y) }\{ \wt X_t' =  \wt X_t'' \text{ for some } t \in (0, \gb) \}.
\] 
The above observations enable us to conclude that
$\underline{p}  >  0$.

For the infinite gasket, if $\gD \in \wt \tcal_n$, then $2^n \gD \in \wt \tcal_0$ and the scaling property of Brownian motion on the infinite gasket gives us that for any $\gD \in \wt \tcal_n$
\[
\begin{split}
& \inf_{x, y \in \gD} \prob^{(x, y) }\{ \wt X_t'  =  \wt X_t''  \text{ for some } t \in (0, 5^{-n} \gb) \}\\  
& \quad = \inf_{x, y \in 2^n \gD} \prob^{(x, y) }\{ \wt X_t'  =  \wt X_t''  \text{ for some } t  \in (0,\gb ) \}. \\ 
\end{split}
\]
Therefore, for any $\gD \in \wt \tcal_n$ and any $x, y \in \gD$, 
\begin{equation}\label{key_lowerbound}
 \prob^{(x, y) }\{ \wt X_t'  =  \wt X_t'' \text{ for some } t \in (0, 5^{-n} \gb) \} \ge \underline{p}.
 \end{equation}
\end{proof}

\begin{corollary}
The Brownian motions $\wt X$ and $X$ on the infinite and finite gaskets
both satisfy Assumption~\ref{ass:immediate_hit}.
\end{corollary}

\begin{proof}
By Theorem~\ref{l:lower_bdd_hitting_prob} and the
Blumenthal zero-one law, we have for
two independent Brownian motions $\wt X'$ and $\wt X''$ on $\wt G$
and any point $(x,x) \in \wt G \times \wt G$ that
\[
\prob^{(x, x) }\{\text{for all $\epsilon > 0$,  $\exists \ 0 < t < \epsilon$ such that
$\wt X_t'  =   \wt X_t''$  }\} = 1.
\]
Lemma~\ref{lem:cont} then gives the claim for $\wt X$.  The proof
for $X$ is similar.
\end{proof}

\section{Instantaneous coalescence on the gasket} 

We will establish the following three results in this section
after obtaining some preliminary estimates.

\begin{theorem}[Instantaneous Coalescence]\label{thm_main1_gasket} 
(a) Let $\Xi$ be the set-valued coalescing Brownian motion process  on $\wt G$
with $\Xi_0$ compact.  Almost surely, $\Xi_t$ is a finite set for all $t>0$.

\noindent
(b) The conclusion of part (a) also holds for the 
set-valued coalescing Brownian motion process on $G$
\end{theorem}

\begin{theorem}[Continuity at time zero]\label{thm_main2_gasket}
(a) Let $\Xi$ be the set-valued coalescing Brownian motion process  on $\wt G$
with $\Xi_0$ compact. Almost surely, $\Xi_t$ converges to  $\Xi_0$ as $t \downarrow 0$.

\noindent
(b) The conclusion of part (a) also holds for the 
set-valued coalescing Brownian motion process on $G$.
\end{theorem}

\begin{theorem}[Instantaneous local finiteness]\label{thm:discrete_gasket}
Let $\Xi$ be the set-valued coalescing Brownian motion process on $\wt G$
with $\Xi_0$ a possibly unbounded closed set. 
Almost surely, $\Xi_t$ is a locally finite set for all $t>0$.
\end{theorem}

\begin{lemma}[Pigeon hole principle]\label{l:pigeon_hole} 
Place $M$ balls in $m$ boxes and
allow any two balls to be paired off together 
if they belong to the same box. Then, the maximum number of disjoint pairs of balls possible is at least $(M-m)/2$.
\end{lemma}

\begin{proof}
Note that in an optimal pairing there can be at most one unpaired 
ball per box. It follows that the number of paired balls is at least $M - m$
and hence the number of pairs is at least $(M - m)/2$.   \end{proof}

Define the $\eps$-fattening  of a set $A \subseteq  \wt G$ 
to be the set $A^\eps: = \{ y \in \wt G: \exists  x \in A, |y-x| < \eps \}$.  Define the $\eps$-fattening of a set $A \subseteq G$ in $G$
similarly.  Recall the constants 
$\underline{p}$ and $\beta$ from Theorem~\ref{l:lower_bdd_hitting_prob}. Set $\gc := 1/(1 - \underline{p} /5) >1$. Given a finite subset $A$ of $\wt G$ or $G$ and a time-interval $I \subseteq \mathbb R_+$, define the random variable  $\rcal(A; I)$ to be the {\em range} of the set-valued coalescing process $\Xi$ in the finite or the infinite gasket during time $I$ with initial state $A$; that is, 
\[ 
\rcal(A;I) :=  \bigcup_{ s \in I} \Xi_s.
\]
Define a stopping time for the same
process $\Xi$ by $\tau^A_{m}  := \inf \{ t : \# \Xi_t \le m  \}$. 

\begin{lemma}\label{l:gasket_main}
(a) Let $\Xi$ be the set-valued coalescing Brownian motion process in the infinite gasket with $\Xi_0 = A$,
where $A \subset  \wt G$ of cardinality $n$ such that $A^\eps $ for some $\eps > 0$ is contained in an extended triangle $\gD^e$ of $\G$.   Then, there exist constants $C_1$ and $C_2$ which may depend on $\eps$ but are independent of $A$ such that
\begin{equation}
\label{eq:gasket_unbdd}
\begin{split}
&\prob \left \{ \tau^A_{\lceil n\gc^{-1} \rceil } >  25 \beta n^{-\log_3 5}  \text{\em or }  \rcal(A, [0,  \tau_{\lceil n \gc^{-1} \rceil }] ) \not \subseteq A^{\eps n^{- (1/6)\log_3 5}}  \right\} \\
& \quad \le C_1 \exp ( - C_2 n^{ 1/3} ). \\
\end{split}
\end{equation}
(b) The same inequality holds for the  set-valued coalescing coalescing Brownian motion process in the finite gasket.
\end{lemma}

\begin{proof}
(a) For any integer $b \ge 1$, the set $A$ can be covered by at most $2 \times 3^b$  $b$-triangles. Put
\[ b_n := \max \{b :  2 \times 3^b \le n/2 \},\]
or, equivalently,
\[
b_n =\lfloor \log_3 (n/4) \rfloor.
\]
By Lemma~\ref{l:pigeon_hole}, at time $t=0$ it is possible to form at least $n/2 - n/4 = n/4$  disjoint pairs of particles, where two particles are deemed eligible to form a pair if they belong to the same $b_n$-triangle.  Fix such an (incomplete) pairing of particles. Define a new ``partial'' coalescing system involving $n$ particles, where a particle is only allowed to coalesce with the one it has been paired up with and after such a coalescence occurs the two partners in the pair both follow the path of the particle having the lower rank among the two.  Evidently this new system  is same as the coalescing system in the marked space where two particles have the same mark if and only if they have been paired up. From the discussion in Subsection \ref{subsec:coupling} the number of surviving distinct particles in this partial coalescing system stochastically dominates the number of surviving particles in the original coalescing system.

By Theorem~\ref{l:lower_bdd_hitting_prob}, the probability 
that a pair in the partial coalescing system coalesces before time $t_n := \beta 5^{-b_n} $  is at least $\underline{p}$, independently of the other pairs. Thus, the number of coalescence by time $t_n$ in the partial coalescing system stochastically dominates a random variable that is distributed as
the number of successes in $n/4$ independent Bernoulli trials
with common success probability  $\underline{p}$.
By Hoeffding's inequality, the probability that a random variable with
the latter distribution takes a value $n \underline{p}/5$ or greater is
at least $1 - e^{ - C'_1 n}$ for some constant $C'_1>0$. Thus,
the probability that the number of surviving particles in the 
original coalescing system drops below $\lceil (1 - \underline{ p}/5) n \rceil =  \lceil n\gc^{-1} \rceil $ by time $t_n \le 25 \gb n^{-\log_3 5}$
is at least  $1 - e^{-C'_1 n}$.

 From Corollary \ref{l:supbound}(a)  and the fact that  during a fixed time interval the maximum displacement of particles in the coalescing system is always bounded by the maximum displacement of independent particles
starting from the same initial configuration, the probability that
over a time interval of length $25 \gb n^{-\log_3 5}$ one of the coalescing particles has moved more than a distance $\eps n^{- (1/6)\log_3 5}$ from
its original position is bounded by
\[
\begin{split} 
& 2 n c_1 \exp\Big ( - c_2 (  (\eps n^{- (1/6)\log_3 5})^{d_w}(25 \gb n^{-\log_3 5})^{-1})^{ 1/ (d_w - 1)}\Big )\\
& \quad \le 2 \exp\Big (  \log n - C'_2 (n^{ (1/2)\log_3 5})^{ 1/ (d_w - 1)}\Big )\\
& \quad \le C_1 \exp ( - C_2 n^{ (1/4)\log_3 5} ) \\
& \quad  \le  C_1 \exp ( - C_2 n^{1/3} ). \\
\end{split}
\]
(b) The proof is identical to part (a). It uses Corollary \ref{cor:lower_bdd_hitting_prob} in place of Theorem \ref{l:lower_bdd_hitting_prob} and Lemma \ref{l:supbound}(b) in place of Lemma \ref{l:supbound}(a).
\end{proof}

\begin{lemma}\label{l:gasket_intermediate}
(a) Let $\Xi$ be the set-valued coalescing Brownian motion process in the infinite gasket with $\Xi_0 = A$.  Fix $\eps>0$. Set  $\nu_i:=\eps \gamma^{-(1/6)\log_3 5 \times  i}$ and $\eta_i = 25 \beta \gc^{-i \log_3 5}$
for $i \ge 1$.
There are positive constants $C_1=C_1(\eps)$ and $C_2 = C_2(\eps)$ such that
 \begin{align*}
\prob &\left \{ \tau^A_{\lceil \gc^k \rceil } > \sum_{i=k+1}^m \eta_i \ \text{\em or }  \rcal(A; [0, \tau^A_{\lceil \gc^k \rceil} ]) \not \subseteq  (A)^{ \sum_{i=k+1}^m \nu_i}\right \}
\le \sum_{i=k+1}^{m}C_1 \exp ( - C_2 \gc^{  i/3} ),
 \end{align*}
uniformly for all sets $A$ of cardinality $\lceil \gc^m \rceil$ such that
the fattening $A^{\sum_{i=k+1}^m \nu_i}$ is contained in some extended triangle $\gD^e$ of $\G$.
 
\noindent
(b)  The analogous inequality holds for the  set-valued coalescing Brownian motion process in the finite gasket.
\end{lemma}

\begin{proof}
Fix an extended triangle $\gD^e$ of the infinite gasket and a  set $A$ such that $\# A = \lceil \gc^m \rceil$
and $A^{\sum_{i=k+1}^m \nu_i} \subseteq \gD^e$. We will prove the bound by induction on $m$. By the 
strong Markov property and Lemma \ref{l:gasket_main}, we have,
using the notation 
$A_{\tau,m-1} := \Xi_{\tau^A_{\lceil \gc^{m-1} \rceil}}$,
\[
\begin{split}
&\prob \Big \{ \tau^A_{\lceil \gc^k \rceil } > \sum_{i=k+1}^m \eta_i  \text{ or }  \rcal(A; [0, \tau^A_{\lceil \gc^k \rceil} ])  \not \subseteq  (A)^{ \sum_{i=k+1}^m \nu_i}\Big \}\\
&\quad \le  
\prob \Big \{ \tau^A_{\lceil \gc^{m-1} \rceil } > \eta_m \ \text{ or }   
 \rcal(A; [0, \tau^A_{\lceil \gc^{m-1} \rceil} ])  \not \subseteq  A^{ \nu_m} \Big \} \\
& \qquad + \E \Big [ 1\Big \{ A_{\tau,m-1}  \subseteq  A^{ \nu_m}  \Big \} \\
& \quad \qquad \times \prob \Big \{ \tau^{A_{\tau,m-1}}_{\lceil \gc^{k} \rceil} > \sum_{i=k+1}^{m-1} \eta_i   \text{ or }  \rcal (A_{\tau,m-1}; [0, \tau^{A_{\tau,m-1}}_{\lceil \gc^k \rceil }] ) \not \subseteq  A^{ \sum_{i=k+1}^{m-1} \nu_i} \Big \} \Big ]\\
& \quad \le C_1 \exp ( - C_2 \gc^{ m/3 } ) \\ 
& \qquad  + \sup_{ \begin{smallmatrix} A_1:  |A_1|  =  \lfloor \gc^{m-1} \rfloor, \\
A_1 \subseteq  A ^{ \nu_m} \end{smallmatrix}} \prob \left \{ \tau^{A_1}_{\lceil \gc^k \rceil } > \sum_{i=k+1}^{m-1} \eta_i  \text{ or } \rcal(A_1; [0, \tau^{A_1}_{\lceil \gc^k \rceil} ]) \not  \subseteq A_1^{ \sum_{i=k+1}^{m-1} \nu_i}\right \}.
 \end{split}
\]
Since $(A^{\nu_m})^{\nu_{m-1}} \subseteq A^{\nu_m + \nu_{m-1} } \subseteq \gD^e$,  the second term on the last expression can  be bounded similarly as
\[
\begin{split}
&\sup_{ \begin{smallmatrix} A_1:  |A_1|  =  \lfloor \gc^{m-1} \rfloor, \\ 
A_1 \subseteq  A ^{ \nu_m} \end{smallmatrix}} \prob \left\{ \tau^{A_1}_{\lceil \gc^k \rceil } > \sum_{i=k+1}^{m-1} \eta_i  \text{ or } \rcal(A_1; [0, \tau^{A_1}_{\lceil \gc^k \rceil} ])  \not  \subseteq A_1^{ \sum_{i=k+1}^{m-1} \nu_i}\right \} \\
&\le C_1 \exp ( - C_2 \gc^{ (m-1)/3} ) \\
&  + \sup_{ \begin{smallmatrix} A_2:  |A_2|  =  \lfloor \gc^{m-2} \rfloor, \\
A_2 \subseteq  A^{ \nu_m +\nu_{m-1}} \end{smallmatrix}} \prob \left \{ \tau^{A_2}_{\lceil \gc^k \rceil } > \sum_{i=k+1}^{m-2} \eta_i  \text{ or } \rcal(A_2; [0, \tau^{A_2}_{\lceil \gc^k \rceil} ])  \not  \subseteq A_2^{ \sum_{i=k+1}^{m-2} \nu_i}\right \}. \\
 \end{split}
\]
 Iterating the above argument, the assertion follows.
 
\noindent
(b) Same as part (a).
\end{proof}

\begin{proof}[Proof of Theorem \ref{thm_main1_gasket}]

 (a) We may assume that $Q := \Xi_0$ is infinite, since otherwise there is nothing to prove.
By scaling, it is enough to prove the theorem when $Q$ is contained in $G$.  Let $Q_1 \subseteq Q_2 \subseteq \ldots \subseteq Q$ be a sequence of finite sets such that $\# Q_m = \lceil \gc^m \rceil$ and  $Q$ is the closure
 $\bigcup_{m=1}^\infty Q_m $. By assigning suitable rankings to a system of independent particles starting from each point in $\bigcup_{m=1}^\infty Q_m$, we can obtain coupled set-valued coalescing processes $\Xi^1, \Xi^2, \dots$ and $\Xi$    
with the property that $\Xi_0^m = Q_m$, $\Xi_0 = Q$, and for each $t>0$,
\[ 
\Xi_t^1 \subseteq \Xi_t^2 \subseteq \ldots \subseteq \Xi_t
\]
and 
$\Xi_t$ is the closure of $\bigcup_{m=1}^\infty \Xi_t^m$.
 
Fix $\eps>0$ so that $Q^{\eps \sum_{i=0}^\infty \gamma^{-(1/6)\log_3 5 \times i} } $ is contained in the extended triangle corresponding to $G$. 
Set  $\nu_i:=\eps \gamma^{-(1/6)\log_3 5 \times  i}$ and $\eta_i := 25 \beta \gc^{-i \log_3 5}$.
Fix $t>0$. Choose $k_0$ so that $\sum_{i=k_0+1}^\infty \eta_i \le t$. 
By Lemma \ref{l:gasket_intermediate} and the fact that $s \mapsto \# \Xi_s^m$ is non-increasing, we have, for each $k \ge k_0$, 
  \[ 
  \prob \left \{  \# \Xi_t^m \le  \lceil \gc^{k} \rceil  \right \} 
  \ge 1- \sum_{i=k+1}^{m}C_1 \exp ( - C_2 \gc^{  i/3} ). 
  \]
 
 By the coupling, the sequence of events 
 $\{  \# \Xi_t^m \le  \lceil \gc^{k} \rceil  \} $ decreases to the event 
 $\{  \# \Xi_t \le  \lceil \gc^{k} \rceil  \}$. 
 Consequently,  letting $m \to \infty$, we have, for each $k \ge k_0$,
 \[ \prob \left \{  \# \Xi_t  \le \lceil \gc^{k} \rceil  \right \} 
 \ge 1 - \sum_{i=k+1}^{\infty}C_1 \exp ( - C_2 \gc^{ i/3} ). 
 \]
Finally letting $k \to \infty$, we conclude that
\[ 
\prob \left \{  \# \Xi_t < \infty  \right \} = 1.
\]
 
\noindent
(b)  Same as part (a).
 \end{proof}
 
\begin{proof}[Proof of Theorem \ref{thm_main2_gasket}]
 
 (a) Assume without loss of generality that $Q := \Xi_0$ is infinite  and contained in the $1$-triangle that contains the origin. By Theorem \ref{thm_main1_gasket}, $\Xi_t$ is almost surely  finite and hence it can be considered as a random element  in $(\kcal, d_H)$.   It is enough to prove that 
$\lim_{t \downarrow 0} d_H(\Xi_t, \Xi_0) = 0$ almost surely.

 Let $Q_1 \subseteq Q_2 \subseteq \cdots$ be a nested sequence of finite approximating sets of $Q$ chosen as 
 in the proof of Theorem \ref{thm_main1_gasket}, and let
$\Xi^m$ be the corresponding coupled sequence of set-valued processes. 
 
Fix $\gd>0$.
Choose $m$ sufficiently large that $Q \subseteq Q_m^{\delta/2}$. By the right-continuity of the finite coalescing process, we have 
\[ \lim_{t \downarrow 0} d_H(\Xi_t^m, Q_m ) \to 0 \quad a.s.\]
 Thus, with probability one, $(\Xi_t^m)^{\delta/2} \supseteq Q_m$  when $t$ is sufficiently close to  $0$.
 But, by the  choice of $Q_m$, with probability one,
 \begin{equation}\label{eq:one_side_ineq}
  (\Xi_t^m)^{\delta} \supseteq (Q_m)^{\delta/2} \supseteq Q
  \end{equation}
  for $t$ sufficiently close to  $0$.

 Conversely, choose $\eps >0$ sufficiently small so that 
 $ \sum_{i=}^\infty \nu_i < \delta/2$ where $\nu_i$ is defined as in Lemma \ref{l:gasket_intermediate}. Set 
 $s_k := \sum_{i=k+1}^\infty \eta_i  \sim C\gc^{ - k \log_3 5}$. From Lemma \ref{l:gasket_intermediate}, we have
 \begin{align} 
 & \prob \Big \{  \rcal(Q_m; [0, s_k])  \not \subseteq  (Q)^{\gd}\Big \} \notag \\ 
&\le \prob \Big \{ \tau^{Q_m}_{\lceil \gc^k \rceil } > \sum_{i=k+1}^m \eta_i \ \text{ or }  \rcal(Q_m; [0, \tau^{Q_m}_{\lceil \gc^k \rceil }])  \not \subseteq  (Q)^{\gd/2}\Big \} \notag \\ 
&+ \prob \Big \{\text{max displacement of $\lceil \gc^k \rceil$ independent particles in } [0, s_{k-1}]   > \gd/2 \Big \} \notag \\
&\le \sum_{i=k+1}^{m}C_1 \exp ( - C_2 \gc^{  i/3} ) + C_1'\lceil \gc^k \rceil \exp ( - C'_2 \gc^{k} )  \notag \\ 
\label{eq:thm_2_ineq}
&\le C_3 \exp ( - C_2 \gc^{  k/3} ).
\end{align}      
By Theorem \ref{thm_main1_gasket}  $\# \Xi_{s} < \infty $ almost surely, 
and hence $\Xi_{s}^m = \Xi_{s}$ for all $m$ sufficiently large almost surely.  
Therefore, by letting $m \to \infty$ in \eqref{eq:thm_2_ineq}, we obtain
\[  \prob \Big \{   \rcal(Q; [0, s_k])  \not \subseteq  (Q)^{\gd}\Big \} \le C_3 \exp ( - C_2 \gc^{  k/3} ).\]
Letting $k \to \infty$, we deduce that,  with probability one, 
\[ \Xi_t \subseteq Q^\delta \] 
 for $t$ sufficiently close to  $0$. Combined with \eqref{eq:one_side_ineq}, 
 this gives the desired claim.
 \end{proof}
 
 \begin{proof}[Proof of Theorem \ref{thm:discrete_gasket}]
 By scaling, it suffices to show that  almost surely, the set $\Xi_t \cap G$ is finite for all $t>0$.   Fix any $0< t_1 < t_2$. We will show that almost surely, the set $\Xi_t \cap G$ is finite for all $t \in [t_1, t_2]$. 
 
Set $J_{0,1} := G$. Now for $r \ge 1$, the set $ 2^{r} G \setminus 2^{r-1}G$ can be covered by 
exactly $2 \times 3^{r-1}$ many $0$-triangles that we will denote by  $J_{r,\ell}$ for $1 \le \ell \le 2 \times 3^{r-1}$. The collection $\{J_{r, \ell}\}$ forms a covering of  the infinite gasket.   

Put $Q := \Xi_0$ and let $D$ be a countable dense subset  of $Q$.  Associate each point of $D$ with one of the
(at most two) $0$-triangles to which it belongs. Denote by $D_{r,\ell}$ the subset of $D$ consisting of
particles associated with $J_{r,\ell}$. 
Construct a partial coalescing system starting from $D$ such that two particles coalesce if and only if they collide and both of their initial positions belonged
to the same set $D_{r,\ell}$. Let $(\Xi_t^{r, \ell})_{t \ge 0}$ denote the set-valued coalescing process  consisting of the (possibly empty) subset of the particles associated with $J_{r, \ell}$. 

Note that $(\bigcup_{(r, \ell)}  \Xi_t^{r, \ell})_{ t\ge 0}$ is the set-valued coalescing process in the marked space where two  particles have same mark if and only if both of them  originate from the same  $D_{r, \ell}$. Approximate the set $D$ by a sequence of increasing finite sets. By appealing to the same kind of reasoning as in  Theorem \ref{thm_main1_gasket}, we can find an increasing sequence of set-valued coalescing processes in the original (resp. marked) space starting from this sequence of increasing finite sets which `approximates' the process $(\Xi_t)_{ t\ge 0}$ ( resp. $(\bigcup_{(r, \ell)}  \Xi_t^{r, \ell})_{t \ge 0}$)
 in the limit. Now using the coupling involving finitely many particles given in Subsection \ref{subsec:coupling}  and then passing to the limit, it follows that
\[ \prob \{ \# \Xi_t \cap G  <\infty  \ \forall t \in [t_1, t_2] \} \ge \prob \{ \# \bigcup_{(r, \ell)}  \Xi_t^{r, \ell} \cap G  < \infty  \ \forall  t \in [t_1, t_2] \}.\]
It thus suffices to prove that almost surely,  the set  $G \cap \bigcup_{(r, \ell)} \Xi_t^{r, \ell}$ is finite for all $t \in [t_1, t_2]$. 

Fix $\gD = J_{r, \ell} \in \wt \tcal_0$.  Recall the notation of 
Lemma~\ref{l:gasket_intermediate}. Find  $\eps >0$ such that $ \gD^{ \sum_{i=0}^\infty \nu_i } \subset \gD^e$.   
Let $A_1 \subseteq A_2 \subseteq \ldots$ be  an increasing sequence of sets such that 
$\bigcup_m A_m = D_{r,\ell}$. 
Construct coupled set-valued
coalescing processes $\wt \Xi^1 \subseteq \wt \Xi^2 \subseteq \ldots \subseteq \Xi^{r,\ell}$
such that $\wt \Xi_0^m = A_m$.
Note that by Lemma \ref{l:gasket_intermediate}
\[
\begin{split}
& \prob \Big\{ \Xi_t^{r,\ell} \cap G  \ne \emptyset  \text{ for some } t \in [t_1, t_2]  \Big\} \\
& \quad = \lim_{m\to \infty}\prob \Big\{ \wt \Xi_t^m \cap G  \ne \emptyset  \text{ for some } t \in [t_1, t_2] \Big\} \\
& \quad \le \limsup_{m\to \infty} \prob \Big\{ \tau^{A_m}_{\lceil \gc^r \rceil } > \sum_{i=r+1}^\infty \eta_i \text{ or }  \Xi_{ \tau^{A_m}_{\lceil \gc^r \rceil } }^m \not \subseteq \gD^e \text{ or max displacement }\\
& \qquad \text{of the remaining $\lceil \gc^r \rceil$  coalescing particles in $[ \tau^{A_m}_{\lceil \gc^r \rceil }, t_2] > (r - 3/2)$}\Big\} \\
& \quad \le \limsup_{m\to \infty} \prob \Big\{ \tau^{A_m}_{\lceil \gc^r \rceil } > \sum_{i=r+1}^\infty \eta_i \text{ or }  \Xi_{ \tau^{A_m}_{\lceil \gc^r \rceil } }^m \not \subseteq \gD^e \Big\} \\
& \qquad + \prob\Big\{ \text{max displacement of  $\lceil \gc^r \rceil$  independent  particles in $[0, t_2] > (r - 3/2)$}\Big\}  \\
& \quad \le C_1' \exp(-C_2\gc^{r/3} )  + 2c_1 \lceil \gc^r \rceil \exp\Big ( - c_2 (  (r-3/2)^{d_w}/t_2 )^{ 1/ (d_w - 1)}\Big )\\
& \quad \le  C_3\exp(-C_4\gc^{r/3} ) \\
\end{split}
\]
for some constants $C_3, C_4>0$ that may depend on $t_2$ but are independent of $r$ and $\ell$. The first of the above inequalities follows from the fact that
\[\inf_{x \in J_{r, \ell}, \ y \in G}  |x - y|  \ge 2^{r-1} -1 \ge r-1,\]
which implies that $\gD^e$ is at least at a distance $(r- 3/2)$ away from $G$.

Now  by a union bound, 
\[ \prob \Big\{ \Xi_t^{r, \ell} \cap G  \ne \emptyset  \text{ for some } t \in [t_1, t_2] \text{ and}  \text{ for some } \ell  \Big\} \le 2 \times 3^{r-1} C_3\exp(-C_4\gc^{r/3} ).\]
By the Borel-Cantelli lemma, the events $\Xi_t^{r, \ell} \cap G \ne \emptyset$ 
for some $ t \in [t_1, t_2]$ happen for only finitely  many $(r, \ell)$
almost surely. This combined with the fact that $\# \Xi_t^{r, \ell} < \infty $ for all $t>0$ almost surely gives that
\[ 
\# \bigcup_{(r,\ell)} ( G \cap \Xi_t^{r, \ell} ) < \infty   \text{ for all } t \in [t_1, t_2]
\]
almost surely.
\end{proof}

\section{Instantaneous coalescence of stable particles}

\subsection{Stable processes on the real line and unit circle}
 Let  $X = (X_t)_{t \ge 0}$  be a (strictly) stable process with index $\alpha>1$ on $\mathbb R$.
 The characteristic function of $X_t$ can be expressed as  $\exp(-\Psi(\gl) t)$ where $\Psi(\cdot)$ is called the characteristic exponent and has the form 
 \[ \Psi(\gl) = c |\gl|^\ga \big ( 1 - i \upsilon \mathrm{sgn}(\gl) \tan(\pi \ga/2) \big), \quad \gl \in (-\infty, \infty), i = \sqrt{-1}. \]
 where $c>0$ and $\upsilon \in [-1, 1]$. The L\'{e}vy measure of $\Pi$ is absolutely continuous with respect to Lebesgue measure, with density
 \[ \Pi(dx) = \left\{\begin{array}{lc}c^+ x^{-\ga-1} dx & \text{ if }  x> 0,\\
 c^- |x|^{-\ga-1} dx & \text{ if }  x< 0,
 \end{array} \right.\]
 where $c^+, c^-$ are two nonnegative real numbers such that $\upsilon  = (c^+ - c^-)/(c^++c^-)$. The process is symmetric if $c^+ = c^-$ or equivalently $\upsilon = 0$.
The stable process has the  scaling property
\[ 
X \stackrel{d}{=} (c^{-1/\ga} X_{ct})_{t \ge 0} 
\]
for any $c>0$.
If we put  $Y_{t} := e^{2 \pi i X_{t}}$, then the process $(Y_t)_{t \ge 0}$ is the
 stable process with index $\alpha >1$ on the unit circle $\cir$.

We define the distance between two points on $\cir$ as the length of the shortest path between them and continue to use the same notation $|\cdot|$ as for the Euclidean metric on the real line.

\begin{theorem}[Instantaneous Coalescence] 
\label{thm_main1_stable} 
(a) Let $\Xi$ be the set-valued coalescing stable process on $\mathbb{R}$ with
$\Xi_0$ compact.
Almost surely, $\Xi_t$ is a finite set for all $t>0$.

\noindent
(b) The conclusion of part (a) holds for the set-valued coalescing stable process on $\cir$.
\end{theorem}

\begin{theorem}[Continuity at time zero]
\label{thm_main2_stable}
(a) Let $\Xi$ be the set-valued coalescing stable process on $\mathbb{R}$ with
$\Xi_0$ compact.  Almost surely, $\Xi_t$ converges to $\Xi_0$ as $t \downarrow 0$.

\noindent
(b) The conclusion of part (a) holds for the set-valued coalescing stable process on $\cir$.
\end{theorem}

\begin{theorem}[Instantaneous local finiteness]
\label{thm:discrete_stable}
Let $\Xi$ be the set-valued coalescing stable process on $\mathbb{R}$ with
$\Xi_0$ a possibly unbounded closed set. 
Almost surely, $\Xi_t$ is a locally finite set for all $t \ge 0$.
\end{theorem}

We now proceed to establish hitting time estimates and maximal inequalities
for stable processes that are analogous to those established for
Brownian motions on the finite and infinite gaskets in Section~\ref{S:gasket}.
With these in hand, the proofs of Theorem~\ref{thm_main1_stable} 
and Theorem~\ref{thm_main2_stable} follow along similar, but simpler, lines
to those in the proofs of the corresponding results for the gasket
(Theorem~\ref{thm_main1_gasket} and Theorem~\ref{thm_main2_gasket}), and
so we omit them.  However, the proof of Theorem~\ref{thm:discrete_stable}
is rather different from that of its gasket counterpart (Theorem~\ref{thm:discrete_gasket}),
and so we provide the details at the end of this section.

\begin{lemma} 
Let $Z = X' - X''$ where $X'$ and $X''$ are two independent copies of $X$,
so that $Z$ is a symmetric stable process with index $\alpha$.
For any $0< \delta < \beta$,
\[ \prob^z\{ Z_t =0 \text{ { for some} } t \in (\delta , \beta)\} > 0.\]
\end{lemma}
\begin{proof}
The proof follows from \cite[Theorem 16]{MR1406564} which says that the single points are not essentially polar for the process $Z$, the fact that $Z$ has a continuous symmetric transition density with respect to Lebesgue measure, and the Markov property of the $Z$.
 \end{proof}
 
 It is well-known that symmetric stable process $Z$ on $\real$ with
 index greater than one hit points (see, for example, \cite[Chapter VIII, Lemma 13]{MR1406564}). Thus there exists a $0< \beta< \infty $ so that
 \[ 0< \prob^1\{  Z_t = 0 \text { for some } t \in (0, \beta) \}=: \underline{p} \text{ (say)}. \]
 By scaling,
  \[ \prob^\eps \{  Z_t = 0 \text { for some } t \in (0, \beta\eps^{\alpha}) \}= \underline{p}. \]
 
 \begin{lemma}
Suppose that $X'$ and $X''$ are two independent stable processes on $\real$ starting at $x'$ and $x''$.  For any $\eps>0$,
\[  \inf_{|x' -x''| \le \eps} \prob\{ X_t' = X_t'' \text{ for some } t \in (0, \beta\eps^{\alpha}) \} = \underline{p}.\]
 \end{lemma}

Since $X_t' = X_t''$ always implies that $\exp(2 \pi i X_t' ) = \exp (2 \pi iX_t'')$ (but converse is not true), we have the following corollary of the above lemma.
\begin{corollary}
 If $Y'$ and $Y''$ are two independent stable processes on $\cir$ starting at $y'$ and $y''$, then  for any $\eps>0$
\[  \inf_{|y' -y''| \le 2 \pi \eps} \prob\{ Y_t' = Y_t'' \text{ for some } t \in (0, \beta\eps^{\alpha}) \} \ge \underline{p}.\]
 \end{corollary}

\begin{lemma} [\cite{MR1406564}]\label{lem:stable_tail}
Suppose that $X$ is an $\alpha$-stable process on the real line. 
There exists a constant $C>0$ such that 
\[ \prob^0 \left\{ \sup_{ 0 \le s \le 1} |X_s|  > u \right \} \le C u^{-\alpha}, \quad u \in \real_+. \]
\end{lemma}

\begin{corollary}\label{cor:sup_bound_stable}
(a) Let $X^1, X^2, \ldots, X^n$ be independent stable processes of index $\ga > 1$ on $\real$ starting from $x^1, x^2, \ldots, x^n$ respectively. Then for each $x \in \real_+$ and $t > 0$,
\[   \prob \Big \{ \sup_{ 0 \le s \le t} |X^i_s -x^i |  > u \text{ for some } 1 \le i \le n \Big \} \le Cn t u^{-\ga}.\]
(b) The same bound holds for $n$ independent stable processes on $\cir$ when $u < \pi$.
\end{corollary}

Again we set $\gc := 1/(1 - \underline{p} /5) >1$.   Fix $(\alpha-1)/2 < \eta < \alpha-1$ and define $h := 1 - (1+\eta)/\ga > 0$. Recall the definitions of $\tau^A_m$ and $\rcal(A;I)$.

\begin{lemma}\label{l:stable_main}
Fix $0 < \eps \le 1/2$. 

\noindent
(a) 
There is a constant $C_1 = C_1(\eps)$
such that $\Xi$ be  a set-valued coalescing stable
process in $\real$ with $\Xi_0 = A$, then
\begin{align} \label{eq:stable_real}
&\prob \left \{ \tau^A_{\lceil n\gc^{-1} \rceil } >  \beta (2\ell/n)^{\alpha}  \ \text{ \em or }  \rcal(A, [0,  \tau_{\lceil n \gc^{-1} \rceil }] ) \not \subseteq A^{\eps \ell n^{-h}}  \right\} 
\le C_1 n^{-\eta},
\end{align}
where $n = \# A$ and $\ell/2$ is the diameter of $A$.

\noindent
(b) Let $\Xi$ be  the set-valued coalescing process in $\cir$ with $\Xi_0 = A$, 
where $A$ has cardinality $n$. Then there exists constant $C_1 = C_1(\eps)$, independent of $A$, such that
\begin{align*} 
&\prob \left \{ \tau^A_{\lceil n\gc^{-1} \rceil } >  \beta (2/n)^{\alpha}  \ \text{ \em or }  \rcal(A, [0,  \tau_{\lceil n \gc^{-1} \rceil }] ) \not \subseteq A^{\eps n^{-h}}  \right\} 
\le C_1 n^{-\eta}.
\end{align*}
\end{lemma}

\begin{proof}
(a) Note that $A^{\eps\ell} \subseteq [a -\ell/2, a+\ell/2]$ for some $a \in \mathbb{R}$, and this
interval can be divided into $n/2$ subintervals of length $2\ell/n$. We follow closely the proof of Lemma \ref{l:gasket_main}. By considering a suitable partial coalescing particle system consisting of at least $n/4$ pairs of particles  where a pair can only coalesce if they have started from the same subinterval, we have that the number of surviving particles in the original coalescing system is at most $ \lceil \gc^{-1} n \rceil$ within time $t_n:= \beta (2\ell/n)^{\alpha}$ with error probability bounded by $\exp(-C_1' n)$.

 By Corollary \ref{cor:sup_bound_stable}, the maximum displacement of  $n$ independent stable  particles on $\real$ within time $t_n$ is at most 
\[ \eps (t_n)^{1/\alpha}  n^{ (1+\eta)/\alpha} =   2\beta^{1/\ga} \eps \ell n^{ -1+ (1+ \eta)/\ga} = 2\beta^{1/\ga} \eps  \ell n^{ -h}\]
with error probability at most  $c_2n^{-\eta}$.

(b) The proof for part (b) is similar.
\end{proof}

Using strong Markov property and  Lemma \ref{l:stable_main} repetitively as we did in the proof of Lemma \ref{l:gasket_intermediate}, we can obtain the following lemma. We omit the details.

\begin{lemma}\label{l:stable_intermediate}
Let $0 < \eps \le 1/2, \ell >0$ be given. Let  $\nu_i:=\eps \gamma^{- h i}$ and $\eta_i := \beta 2^{\ga} \gc^{-\ga i }$. 

\noindent
(a) Given a finite set $A \subset \real$, let $\Xi$ denote the set-valued coalescing stable process in $\real$ with $\Xi_0 =A$ .  Then, there exist constants $C_2 = C_2(\eps)$ such that 
 \begin{align*}
\prob &\Big \{ \tau^A_{\lceil \gc^k \rceil } > \ell^\ga \sum_{i=k+1}^m \eta_i \ \text{\em or }  \rcal(A; [0, \tau^A_{\lceil \gc^k \rceil}] )  \not \subseteq  (A)^{\ell \sum_{i=k+1}^m \nu_i}\Big \}\le C_2  \gc^{ -\eta k},
 \end{align*}
uniformly over all  sets $A$ such that $A \subseteq [a -\ell/4, a+\ell/4]$  for some $a \in \real$ and  $\# A = \lceil \gc^m \rceil$.
 
\noindent
(b)  Given a finite set $A \subset \cir$, let $\Xi$ denote the set-valued coalescing stable process in $\cir$ with $\Xi_0 =A$ .  Then, there exist constants $C_2 = C_2(\eps)$ such that 
 \begin{align*}
\prob &\Big \{ \tau^A_{\lceil \gc^k \rceil } > \sum_{i=k+1}^m \eta_i \ \text{\em or }  \Xi_{\tau^A_{\lceil \gc^k \rceil}}  \not \subseteq  (A)^{ \sum_{i=k+1}^m \nu_i}\Big \}\le C_2  \gc^{ -\eta k},
 \end{align*}
uniformly over all  sets $A \subseteq \cir$ such that $\# A = \lceil \gc^m \rceil$.
\end{lemma}


\begin{proof}[Proof of Theorem \ref{thm:discrete_stable}]
By scaling, it is enough to show that  for each $0< t_1<t_2<\infty$, almost surely, the set  $\Xi_t \cap [-1,1]$ is finite for each $t \in [t_1, t_2]$. Set $d := 2/\eta$. For $r\ge 1$, define  
\[
J_{r,1} := \Big [ - \sum_{j=1}^{r} j^{d}, -\sum_{j=1}^{r-1} j^{d} \Big) 
 \ \ \text{ and }   
 \ \ J_{r,2} := \Big [\sum_{j=1}^{r-1} j^{d}, \sum_{j=1}^{r} j^{d} \Big).
\] 
 Then the collection $\{J_{r,i}\}_{ r \ge 1, i=1,2}$ forms a partition of the real line into bounded sets.   Note that $ \inf_{x \in [-1,1], y \in J_{r,i}} |x- y| \asymp r^{d+1}$ as $r \to \infty$.

Let $D$ be a countable dense subset  of $Q$.  Run a partial coalescing system starting from $D$ such that two particles coalesce if and only if they collide and both belonged initially to the same $J_{r,i}$. Let $(\Xi_t^{r,i})_{t \ge 0}$ denote the set-valued coalescing process  consisting of the (possibly empty) subset of the particles starting from $D \cap J_{r,i}$. 
By arguing similarly as in the proof of Theorem \ref{thm:discrete_gasket},  it  suffices to prove that the set  $[-1,1] \cap  \Xi_t^{r, i}$ is empty  for all $t \in [t_1, t_2]$ for all but finitely many pairs $(r,i)$ almost surely.

Fix a pair $(r, i)$. Find  $\eps >0$ such that $  \sum_{i=0}^\infty \nu_i  \le 1/2$ which implies that 
$(J_{r,i})^{ \sum_{i=0}^\infty \nu_i} \subseteq (J_{r,i})^{r^d}$.   
Let $A_1 \subseteq A_2 \subseteq \ldots$ be  an increasing sequence of finite sets
such that for $\bigcup_m A_m = D \cap J_{r, i}$.  Let $\wt \Xi^m$ be a coalescing
set-valued stable processes such that $\wt \Xi_0^m = A_m$ and couple
these processes together so that 
$\wt \Xi_t^1 \subseteq \wt \Xi_t^2 \subseteq \ldots \subseteq \Xi_t^{r,i}$.
Set $b = b(r) := (2/\eta) \lceil \log_\gc r \rceil $. Note  that by Lemma~\ref{l:stable_intermediate}, Corollary~\ref{cor:sup_bound_stable}, and  the fact that there exists $c_1>0$ such that for all $r$ sufficiently large
\[ 
\inf_{x \in [-1,1], y \in J_{r,i}} |x- y| - r^d \ge c r^{d+1}, 
\] 
we can write
\[
\begin{split}
& \prob \Big\{ \Xi_t^{r,i} \cap [-1, 1]  \ne \emptyset \text{ for some } t \in [t_1, t_2] \Big\} \\ 
& \quad = \lim_{m\to \infty}\prob \Big\{ \wt \Xi_t^m \cap [-1,1] \ne \emptyset \text{ for some } t \in [t_1, t_2] \Big\} \\
& \quad \le \limsup_{m\to \infty} \prob \Big\{ \tau^{A_m}_{\lceil \gc^b \rceil } > \sum_{i=b+1}^\infty \eta_i \text{ or }  \wt \Xi_{ \tau^{A_m}_{\lceil \gc^b \rceil } }^m \not \subseteq (J_{r,i})^{r^d} \text{ or max displacement }\\
& \qquad \text{of the remaining $\lceil \gc^b \rceil$  coalescing particles in $[ \tau^{A_m}_{\lceil \gc^b \rceil }, t_2]  > cr^{d+1}$}\Big\} \\
& \quad \le \limsup_{m\to \infty} \prob \Big\{ \tau^{A_m}_{\lceil \gc^b \rceil } > \sum_{i=r+1}^\infty \eta_i \text{ or }  \wt \Xi_{ \tau^{A_m}_{\lceil \gc^b \rceil } }^m \not \subseteq (J_{r,i})^{r^d} \Big\} \\
& \qquad + \prob\Big\{ \text{max displacement of  $\lceil \gc^b \rceil$  independent  particles in $[0, t_2] >cr^{d+1}$}\Big\}  \\
   & \le C_2 \gc^{-\eta b}  + C_3\lceil \gc^b \rceil c r^{-\ga(d+1)}
  \le C_2' r^{-2} + C_3' r^{-\ga} \\
\end{split}
\]
for suitable constants $C_2', C_3'>0$. The proof now follows from the Borel-Cantelli lemma.
\end{proof}

\newcommand{\etalchar}[1]{$^{#1}$}
\providecommand{\bysame}{\leavevmode\hbox to3em{\hrulefill}\thinspace}
\providecommand{\MR}{\relax\ifhmode\unskip\space\fi MR }
\providecommand{\MRhref}[2]{%
  \href{http://www.ams.org/mathscinet-getitem?mr=#1}{#2}
}
\providecommand{\href}[2]{#2}

\end{document}